\documentclass[11pt]{article}
 
\usepackage[margin=1in]{geometry}
\usepackage{amsmath,amsthm,amssymb}
\usepackage{xcolor}
\usepackage[final]{pdfpages}
\usepackage{algpseudocode}
\usepackage{algorithm}

\usepackage{amsmath}    % need for subequations
\usepackage{amssymb}
\usepackage{graphicx}   % need for figures
\usepackage{verbatim}   % useful for program listings
\usepackage{color}      % use if color is used in text
\usepackage{hyperref}   % use for hypertext links, including those to external documents and URLs
\usepackage{enumerate}
\usepackage{xcolor}
\usepackage{slashed}
\usepackage{amsthm}

\usepackage{mathrsfs}
\usepackage{bbold}
\usepackage{lscape}
\usepackage[american]{babel}
\usepackage{mathtools}
\usepackage{geometry}
\usepackage{mathbbol}
\usepackage{nicefrac}
\usepackage{makecell}
\usepackage{enumerate}
\usepackage{tikz}
\usetikzlibrary{shapes.geometric}
\usetikzlibrary{shapes.misc}

\usepackage[final]{showlabels}

\usepackage{float}
\restylefloat{figure}

\usepackage{etoolbox}

\usepackage{ulem}
 \usepackage{cleveref}
\crefformat{footnote}{#2\footnotemark[#1]#3}

\renewcommand{\title}{On the equivalence of semi-discrete Active Flux and Discontinuous Galerkin methods and a comparison of their performance}

\newcommand{\authorOne}{Wasilij Barsukow\footnote{CNRS, Institut de Mathématiques de Bordeaux (IMB), UMR 5251, 351 Cours de la Libération, 33405 Talence, France, Email: wasilij.barsukow@math.u-bordeaux.fr. \emph{Present address:} Imperial College London, IRL2004, Huxley Building, South Kensington Campus, London, United Kindgom}}
\newcommand{\authorTwo}{Christian Klingenberg\footnote{\label{fn:wue}University of Würzburg, Institute of Mathematics, Emil-Fischer-Straße 40, Würzburg, 97074, Germany}}
\newcommand{\authorThree}{Simon Krotsch\cref{fn:wue}}
\begin{document}

\newcommand{\del}{\partial}
\renewcommand{\theta}{\vartheta}
\renewcommand{\phi}{\varphi}
\newcommand{\vecc}[2]{\left ( \begin{array}{c}#1\\#2\\ \end{array}\right )}
\newcommand{\veccc}[3]{\left ( \begin{array}{c}#1\\#2\\#3\\ \end{array}\right )}
\newcommand{\dd}{\mathrm{d}}
\newcommand{\ee}{\mathrm{e}}
\newcommand{\ii}{\mathbb{i}}
\newcommand{\id}{\mathbb{1}}
\newcommand{\atanh}{\,\text{artanh}\,}
\newcommand{\atan}{\arctan}
\newcommand{\back}{\!\!\!}
\newcommand{\bboxed}[1]{\text{\textsc{Conjecture}: }\boxed{\boxed{#1}}}
\renewcommand{\and}{\wedge}
\newcommand{\primitive}{\text{\textsc{primitive}}}
\newcommand{\lint}{\int\limits}
\renewcommand{\div}{\mathrm{div\,}}
\renewcommand{\vec}{\mathbf}
\newcommand{\todo}[1]{{\color{red}TODO: #1}}
\newcommand{\new}[1]{{\color{blue}#1}}
\newcommand{\symtri}{\protect\tikz \protect\node[isosceles triangle,
draw=black, anchor=south, rotate=135, isosceles triangle apex angle=90,
minimum height=0.12cm, inner sep=0, outer sep=0] at (0, 0) {};}

\definecolor{mygreen}{rgb}{0,0.6,0}
\definecolor{mygray}{rgb}{0.5,0.5,0.5}
\definecolor{mymauve}{rgb}{0.58,0,0.82}

\newcommand{\N}{\mathbb{N}}
\newcommand{\Z}{\mathbb{Z}}
\newcommand{\R}{\mathbb{R}}
 
\newtheorem{theorem}{Theorem}[section]
\newtheorem{definition}[theorem]{Definition}
\newtheorem{lemma}[theorem]{Lemma}
\newtheorem{corollary}[theorem]{Corollary}
\newtheorem{ident}{Identification}

\theoremstyle{remark}
\newtheorem{example}[theorem]{Example}
\newtheorem{remark}[theorem]{Remark}

\AtEndEnvironment{example}{%
  \renewcommand{\qedsymbol}{$\triangleleft$}\qed
}

\begin{center} \Large
\title

\vspace{1cm}

\date{\today}
\normalsize

\authorOne, \authorTwo, \authorThree
\end{center}

\begin{abstract}

% #############################################################################################################################
The Active Flux (AF) method employs a globally continuous approximation, like continuous Finite Element methods. This is achieved through the placement of point values at cell interfaces which are shared between adjacent cells. With, on average, $K+1$ degrees of freedom per cell, Active Flux achieves a polynomial approximation of degree $K+1$, while the Discontinuous Galerkin (DG) method uses only polynomials of degree $K$, i.e. one degree less with the same number of degrees of freedom. Despite all the differences, in this paper we show, however, that for linear problems in one and several dimensions as well as---in some sense---for nonlinear ones, semi-discrete AF and DG are the same method. We identify a mapping between their respective degrees of freedom, upon which the updates of these degrees of freedom turn out to agree. On the one hand, AF therefore seems more economical then DG for a given value of the error, and we confirm this in numerical experiments. On the other hand, this is a way to understand superconvergence of DG in a natural way, and we show how Radau polynomials and their zeros appear in the mapping between DG and AF: In the Radau points, AF ''shines through`` as the background high-order scheme behind DG.
% #############################################################################################################################



\end{abstract}

%\tableofcontents

\section{Introduction}

The numerical error of high-order methods decreases as a high power of the discretization length, at least on smooth solutions. The construction of such methods always amounts to an approximation of the solution by a polynomial of high degree; however, it is possible to distinguish different strategies to construct the polynomial. The most important distinction is associated to whether the stencil is compact. If only one degree of freedom is stored per computational cell, as in Finite Difference (FD) or Finite Volume (FV) methods, then a high-degree polynomial can only be constructed by considering sufficiently many neighbours, increasingly far away. Finite Element (FE) methods by construction consider the numerical solution to be a polynomial in every cell, with varying requirements on what happens at cell interfaces. To this end, they store several degrees of freedom in every cell. The evolution of any degree of freedom in some cell depends at most on degrees of freedom in the cells which are immediate neighbours. Such methods shall be referred to as compact.

Compact methods cannot in general be seen as FD methods on a finer mesh, because first of all, the degrees of freedom might not be associated to point values, and second because even then, and even if the point values are equidistant, their update equations are different. Such a method with, say, $N$ degrees of freedom per cell, could in principle be seen as a Finite Difference method on a finer mesh, if one accepts that there are $N$ different update prescriptions and the value at a FD index $i \in \mathbb Z$ is updated according to prescription $i \!\mod N$. Such an interpretation, however, amounts to no simplification. 

While classical Finite Element methods are usually derived from a variational form of the PDE (Galerkin methods), other ways of evolving a given set of degrees of freedom can be conceived. The Active Flux (AF) method, introduced in \cite{vanleer77,eymann13}, specifies directly what the update equations for its degrees of freedom are. These are, classically, a cell average and a shared point value at each cell interface, augmented by further point values or moments to obtain higher than 3rd order accuracy. The shared point values enforce global continuity of the approximation. 

The original Active Flux method employs so-called evolution operators (e.g. \cite{eymann13,fan15,barsukow18activeflux,kerkmann18,barsukow19activeflux,chudzik24,barsukow25afos}) to specify the new value of the degrees of freedom directly, amounting to a single-stage fully discrete method. Such methods have advantages, e.g. a large domain of stability and computational efficiency due to their single-stage nature, but for sufficiently complicated problems general algorithms are unavailable so far. While for linear advection back-tracking of characteristics allows to compute the new point value efficiently, for multi-d problems so far evolution operators are only available for specific equations. Here, we instead focus on the semi-discrete version of Active Flux (\cite{abgrall20,abgrall22}), where only space derivatives are discretized while time is left continuous according to a method-of-lines approach. In the case of the semi-discrete Active Flux method, the update equation of the average is inspired by FV methods, but taking into account the continuity, while the update of the point values is inspired by FD methods. It has, however, also been shown recently (\cite{barsukow25afasfe}) that the semi-discrete Active Flux method can actually be derived from a variational form by considering a biorthogonal set of test and basis functions. The latter property implies that the (global) mass matrix is just identity. This exemplifies once more that there are different derivations leading to the same method.

We are interested in hyperbolic PDEs, in particular conservation laws. It is well-known that some form of stabilization is required for the numerical method to be amenable to explicit integration in time. This can be done either by directly introducing artificial viscosity in the PDE and thus effectively discretizing a regularized parabolic problem instead (e.g. SUPG, see \cite{brooks82}), or by choosing the discretization to have a directional bias (upwinding), an approach inspired by Riemann solvers for FV methods. One of the most prominent examples of Finite Element methods following the latter approach are Discontinuous Galerkin (DG) methods \cite{lesaint74,cockburn89}. They allow for discontinuities at cell interfaces, and use a Riemann solver to compute the flux. Active Flux includes upwinding only in the update of the point values, while continuity of the approximation allows to retain a central update of averages/moments through an integration by parts. 

Here, we intend to compare AF and DG theoretically and experimentally. At first glance it seems that DG, based on a discontinuous approximation and a variational form, must be very different from AF, which approximates the numerical solution continuously with direct updates of the degrees of freedom (mass-matrix-free). As has been mentioned above, AF can actually be derived from a variational form, though. In this work, we show a much more surprising result---AF and DG are in many cases \emph{exactly the same method}. By this we mean that up to a linear mapping between the degrees of freedom of these methods, a DG method with polynomials of degree $K$ amounts to exactly the same update equations for its degrees of freedom as AF based on polynomials of degree $K+1$. Both methods, however, have some freedom (the choice of numerical flux for DG, the choice of the point value update in AF, etc.), and by equivalence we thus mean that there exists an AF-type method that is equivalent to a given DG method for some given PDE. In this sense, we show equivalence of DG and AF for linear and nonlinear systems in 1-d and for linear advection on Cartesian meshes in 2-d.

In its simplest form, this result appears already in \cite{roe17superconvergence} for linear advection in 1-d, but in our opinion is not accorded sufficient importance, even though it is key to understanding the superconvergence of DG. Our present work can be understood as a generalization of this result. We also show why Radau points/polynomials appear naturally in the mapping between DG and AF. The superconvergence of DG thus finds a new interpretation: at Radau points, the DG approximation agrees with the AF approximation, and the underlying higher-order accuracy of AF can be measured in the DG setting.

High order of accuracy does not tell everything about the properties of the method. In particular, with increasing order the computational effort might also increase, possibly in a prohibitive manner. In the second part of this work, we thus also perform comparisons of computational complexity and cost. Equivalence of the two methods means that they have the same number of degrees of freedom per cell, but certain degrees of freedom might be computationally easier to update than others. Active Flux tends to make more use of point values, whose updates do not require expensive numerical quadratures. In multi-d it is not clear to what level equivalence between AF and DG can be established for non-tensorial bases, of which some have been recently proposed for AF in \cite{lechner25}. Some limited, but encouraging results are available concerning superconvergence of DG in these more complicated situations, though, e.g. \cite{krivodonova03,cao25}. We thus complement our theoretical study by a set of careful measurements of run time of AF and DG as function of order of accuracy, degrees of freedom and numerical error. We give some practical guidance concerning the tradeoff between cost and error.

In \cite{roe18}, which carries a title very similar to that of the present work, only the Active Flux method based on evolution operators was compared to DG. While \cite{roe18} finds that this fully discrete version of Active Flux is more efficient than DG, in the present work we find that a similar statement still holds even for the semi-discrete AF, which employs the method of lines and Runge-Kutta time integrators, just as DG. Most importantly, however, the theoretical aspects of equivalence and superconvergence were not addressed in \cite{roe18}, which was entirely devoted to a practical comparison. We recently have also become aware of \cite{abgrall25}, where the relation between AF and DG is studied, but without sufficient attention to the important role played by the upwinding in the point value update/numerical flux in DG, and without making the link to superconvergence of DG.

The paper is structured as follows: Section \ref{sec:cast} gives an overview of AF, DG and the all-important Radau polynomials. Section \ref{sec:equivalence} presents the equivalence and superconvergence results. Section \ref{sec:perfromance} is dedicated to an analysis of the expected  computational complexity and a comparison to experimental measurements.

A computational cell $C_{i}$ in 1-d is $[x_{i-\frac12}, x_{i+\frac12}]$, with $x_i$ its centroid. We assume $\Delta x := x_{i+\frac12} - x_{i-\frac12}$ to be constant. $P^K(I)$ is the space of univariate polynomials of degree at most $K$ on the interval $I$, which we will omit whenever clear. Certain polynomials carry as a superscript their degree (usually, $K$ or $K+1$); we believe that no confusion with powers is reasonably possible. 
Computational cells in 2-d are $C_{ij} := [x_{i-\frac12}, x_{i+\frac12}] \times [y_{j-\frac12}, y_{j+\frac12}]$, with $\Delta y := y_{j+\frac12} - y_{j-\frac12}$, also considered constant. We restrict ourselves here to tensor-product spaces $P^{K,K} \equiv Q^K := P^K \times P^K$.
Indices never denote derivatives.

\section{Cast} \label{sec:cast}

In one spatial dimension, denote by $V^{K+1}$ and $V_\text{br}^K $ the following spaces:
\begin{itemize}
 \item $\displaystyle  V^{K+1} := \{ v \in C^0, v |_{C_i} \in P^{K+1} \,\, \forall i \in \mathbb Z\} $ is the space of continuous, piecewise polynomial (of degree $K+1$) functions (the natural approximation space of AF) and
 \item $ V_\text{br}^K := \{ v, v |_{C_i} \in P^{K}  \,\, \forall i \in \mathbb Z\} $ is the space of piecewise polynomial functions of degree $K$ (the natural approximation space of DG). 
\end{itemize}
Both have $K$ degrees of freedom per cell. We use the letter $q$ for the exact solution of the conservation law under consideration.
To distinguish approximations and degrees of freedom of DG and AF, we denote the latter in capitals, i.e. the approximations are $q_h \in V_\text{br}^K$ and $Q_h \in V^{K+1}$, where the subscript $h$ here, by tradition, merely indicates the fact of being an approximation, since we prefer the notation $\Delta x, \Delta y$ for the grid cell size otherwise. We denote their restrictions onto cell $C_i$ by $q_i \in P^K\left(\left[-\frac{\Delta x}{2}, \frac{\Delta x}{2}\right]\right)$ and $Q_i \in P^{K+1}\left(\left[-\frac{\Delta x}{2}, \frac{\Delta x}{2}\right]\right)$, respectively. 
In multiple dimensions, we consider tensor-product spaces $\displaystyle  V^{K+1} = \{ v \in C^0, v |_{C_{ij}} \in Q^{K+1} \} $ and $ V_\text{br}^K = \{ v , v |_{C_{ij}} \in Q^{K} \} $ instead, with analogous restrictions 
\begin{align*}
q_{ij}&\in Q^K\left(\left[-\frac{\Delta x}{2}, \frac{\Delta x}{2}\right] \times \left[-\frac{\Delta y}{2}, \frac{\Delta y}{2}\right] \right)&Q_{ij} &\in Q^{K+1}\left(\left[-\frac{\Delta x}{2}, \frac{\Delta x}{2}\right]\times \left[-\frac{\Delta y}{2}, \frac{\Delta y}{2}\right]\right)
\end{align*}

 The dependence of $Q_h$ or $q_h$ on time will frequently be omitted, and we will not adapt the notation for the spaces in case of PDE systems, i.e. when $q$ actually has several components.

\subsection{Active Flux}

By Active Flux we mean in the following the semi-discrete arbitrary-order generalization with additional moments (such as \cite{abgrall22,lechner25}), discussed now in one and several space dimensions separately.

\subsubsection{One spatial dimension}\label{ssec:af1d}

The arbitrary-order version of AF with additional moments appeared first in \cite{abgrall20} (for 1-d). The degrees of freedom are the (shared) point values at cell interfaces $Q_{i+\frac12}(t) \simeq q(t, x_{i+\frac12})$ and the moments
\begin{align}
 Q^{(k)}_i \simeq \frac{A_k}{\Delta x} \int_{x_{i-\frac12}}^{x_{i+\frac12}} b_k(x - x_i) q(t, x) \, \dd x \qquad k = 0, \ldots, K-1
\end{align}
where $b_k$ span $P^{K-1}\left(\left[-\frac{\Delta x}{2}, \frac{\Delta x}{2}\right]\right)$, and $A_k$ are normalizations. A possible choice is $b_k = (\frac{2x}{\Delta x})^k$ and $A_k = k+1$, which ensures that $Q^{(k)}_i = 1$ for $q \equiv 1$ if $k$ even. Given $Q_{i\pm\frac12}$ and $\{ Q_i^{(k)} \}_{k = 0, \ldots, K-1}$, in every cell a polynomial $Q_i$ of degree $K+1$ is uniquely determined. It is called the AF reconstruction or the AF approximation. Tthe interpolation of the shared point values implies global continuity: $Q_i(\frac{\Delta x}{2}) = Q_{i+\frac12} = Q_{i+1}(-\frac{\Delta x}{2})$.

The semi-discrete Active Flux method for the $m \times m$ hyperbolic system of conservation laws
\begin{align}
 \del_t q + \del_x f(q) &= 0 & f &\colon \mathbb R^m \to \mathbb R^m \label{eq:conslaw}
\end{align}
then amounts to the following moment update
\begin{align}
 \frac{\dd}{\dd t} Q^{(k)} - \frac{A_k}{\Delta x} \int_{-\frac{\Delta x }{2}}^{\frac{\Delta x}{2}} b_k'(x) f(Q_i(x)) \,\dd x+ A_k\frac{b_k^+ f(Q_{i+\frac12}) - b_k^- f(Q_{i-\frac12})}{\Delta x}   &=0 \label{eq:af2}
\end{align}
where $b_k^\pm := b_k(\pm \Delta x/2)$. Observe that no Riemann fluxes are necessary.
The integral can be evaluated exactly for linear $f$ or by quadrature.
The point value update can, for example, be
\begin{align}
 \frac{\dd}{\dd t} Q_{i+\frac12} + J^+ \frac{\dd}{\dd x} Q_i\Big|_{x = \frac{\Delta x}{2}} + J^- \frac{\dd}{\dd x} Q_{i+1}\Big|_{x = -\frac{\Delta x}{2}} &= 0 \label{eq:af1}
 \end{align}
but there exist other versions (see e.g. \cite{barsukow24afeuler,duan24}) for nonlinear problems. Here, $J := \nabla_q f$ is the flux Jacobian evaluated in $Q_{i+\frac12}$ and $J^\pm$ are the positive/negative parts of $J$ according to the sign of the eigenvalues.

\subsubsection{Two spatial dimensions}

The classical (third-order) AF method in two space dimensions (\cite{eymann13}) employs point values at nodes and at edge midpoints, as well as a cell average. On Cartesian grids (\cite{barsukow18activeflux,kerkmann18}) this leads to 4 node values, 4 edge values and 1 average accessible to a grid cell, and thus naturally to a locally $Q^{2}$ and globally continuous approximation. The point values are shared and there are 4 degrees of freedom per cell: one average, one node value, one value on midpoints of horizontal edges, and one on vertical ones.
As will be seen below, the interplay with DG will suggest new choices of degrees of freedom, that will be introduced then. 

Several arbitrary-order extensions using additional points along the cell interfaces and additional moments have been suggested in \cite{lechner25}. Here, a tensorial approach will appear upon the comparison to DG. It is, however, inefficient from a practical point of view, because the number of degrees of freedom is (asymptotically) about twice the number required to reach the desired order of accuracy. Naive ways to minimize the number of moments involved can lead to a loss of unisolvence. In \cite{lechner25}, the most economical choice was that of serendipity-type degrees of freedom; the reader is referred to this work for further details. Our experimental study is based on this more efficient approach.
 
\subsection{Radau polynomials}

It is well-known that the zeros of Radau polynomials are important in the context of superconvergence of DG. We briefly review how they appear in the context of numerical quadrature and contrast them with Legendre polynomials. Recall that long division of a polynomial $f\in P^{2K+1}(I)$, $K \geq 0$ allows to write it as $f = p \mathscr L^{K+1} + q$, where $p, q \in P^K(I)$ and $\mathscr L^{K+1}$ is a Legendre polynomial of degree $K+1$ chosen such that
\begin{align}
 \int_I b(x) \mathscr L^{K+1}(x) \, \dd x = 0 \qquad \forall b \in P^K(I) \label{eq:legendredef}
\end{align}
over some interval $I$. Observe that these are $K+1$ equations for $K+2$ coefficients of $\mathscr L^{K+1}$, leaving one free for normalization.
But then,
\begin{align}
  \int_I f(x) \, \dd x =  \int_I q(x) \, \dd x
\end{align}
while the point values of $q$ and $f$ agree in the zeros of $\mathscr L^{K+1}$. This is Gauss-Legendre quadrature. $\mathscr L^{K+1}$ does not vanish at the endpoints of $I$. 

Gauss-Radau quadrature arises upon including one of the endpoints of $I$, say $x^*$. Thus, instead of \eqref{eq:legendredef}, the defining relations for the Radau polynomial $\mathscr R^{K+1} \in P^{K+1}(I)$ are
\begin{subequations}
\begin{align}
 \mathscr R^{K+1}(x^*) &= 0 \\
 \int_I b(x) \mathscr R^{K+1}(x) \, \dd x &= 0 \qquad \forall b \in P^{K-1} \label{eq:radaudefint}
\end{align}\label{eq:radaudef}
\end{subequations}
with $K \geq 1$ this time.
Observe that these are again $K+1$ equations for $K+2$ coefficients of $\mathscr R^{K+1}$, leaving one free for normalization. Observe also that orthogonality now happens with respect to polynomials one degree lower, such that one is only interested in writing $f = p \mathscr R^{K+1} + q$ for $p \in P^{K-1}(I)$. Here, $q$ (the remainder of long division of $f$ by $\mathscr R^{K+1}$) is still of degree $K$, and $f$ therefore is only of degree $2K$, one less than for Gauss-Legendre quadrature. The rest of the arguments is entirely analogous. 

Correspondingly, while 
\begin{align}
\int_I b(f - \mathrm{P}f) \, \dd x = 0 \qquad \forall b \in P^K(I)
\end{align}
defines the $L^2$ projection $\mathrm{P}f \in P^K(I)$ of $f \in L^2$, the Gauss-Radau projection $\mathrm{P}_\text{GR}f \in P^K(I)$ is, for $K\geq 1$, defined by
\begin{align}
f(x^*) - \mathrm{P}_\text{GR}f(x^*) &= 0\\
\int_I b\left(f - \mathrm{P}_\text{GR}f\right) \, \dd x &= 0 \qquad \forall b \in P^{K-1}(I)
\end{align}

From \eqref{eq:radaudefint} one is easily led to the ansatz of a linear combination of $\mathscr L^{K+1}$ and $\mathscr L^K$ for $\mathscr R^{K+1}$. The exact coefficients in the linear combination depend on the normalizations chosen for the Legendre and the Radau polynomials and on the interval $I$. With the usual choices, if $x^*$ is the right endpoint of $I$, $\mathscr R^{K+1}$ is a difference of these Legendre polynomials, and if $x^*$ is the left endpoint, it is their sum.

\subsection{Discontinuous Galerkin}

The DG method for the conservation law \eqref{eq:conslaw} is given by the variational form
\begin{align}
 \frac{1}{\Delta x} \int_{-\frac{\Delta x}{2}}^{\frac{\Delta x}{2}} v \frac{\dd}{\dd t} q_i \, \dd x - \frac{1}{\Delta x} \int_{-\frac{\Delta x}{2}}^{\frac{\Delta x}{2}}  v'  f(q_i)\, \dd x +  \frac{v^+ \hat f_{i+\frac12} - v^- \hat f_{i-\frac12}}{\Delta x} = 0 \label{eq:dg}
\end{align}
for all $v \in P^K\left(\left[-\frac{\Delta x}{2}, \frac{\Delta x}{2}\right]\right)$, and where $v^\pm := v(\pm \frac{\Delta x}{2})$. The second integral needs to be replaced by a quadrature. With $q_i^\pm := q_i(\pm \frac{\Delta x}{2})$, the numerical flux $\hat f_{i+\frac12}$ is assumed to be a two-point flux
\begin{align}
 \hat f_{i+\frac12} = \hat f(q_i^+, q_{i+1}^-) \label{eq:twopointflux}
\end{align}
The upwind flux, for example, amounts to the choice
\begin{align}
 \hat f_{i+\frac12} = \frac{f(q_i^+) + f(q_{i+1}^-)}{2} - \frac12 (J^+ - J^-)(q_{i+1}^-  -q_i^+)  \label{eq:upwindflux}
\end{align}

The superconvergence properties of DG have been known for a long time and are mentioned, studied and used e.g. in \cite{lesaint74,biswas94,adjerid02}, and there is also extensive literature on e.g. parabolic equations. Concerning recent results on hyperbolic PDEs, we refer the reader to \cite{cheng08,cheng10,yang12superconvergence,guo13,cao14linear1dupwind,cao17linear1d} for the linear 1-d case, to \cite{cao2015linear2dupwind,xu22} for the linear scalar tensor-product 2-d case, to \cite{cao2017linearvariable1d} for variable-coefficient linear problems in 1-d and to \cite{cao18} for nonlinear scalar problems in 1-d, among others. High-order solutions can in certain cases be recovered through post-processing, see e.g. \cite{cockburn03,ryan05}.

The following orders of accuracy for smooth solutions are known for DG with $P^K$ approximation (possibly under suitable assumptions on the grid and the initial data) at least for linear problems:
\begin{itemize}
 \item $2K+1$ for cell averages and numerical fluxes,
 \item $K+2$ at Radau points (zeros of the Radau polynomial) or their generalizations,
 \item $K+2$ between the DG solution and the Gauss-Radau projection of the exact solution (see e.g. \cite{zhang04}),
\end{itemize}
The necessity to consider generalizations of Radau points is connected to choices other than that of the upwind numerical flux (see e.g. \cite{cao17linear1d}).

\section{Equivalence between DG and Active Flux and superconvergence of DG} \label{sec:equivalence}

In this Section, consider $K \geq 1$. Observe that the space $V^{K+1}$ of continuous, piecewise polynomial (of degree $K+1$) functions has the same dimension as the DG space $V_\text{br}^K$ of piecewise polynomial functions of degree $K$: both have $K+1$ degrees of freedom per cell.

\begin{definition} \label{def:equiv}
 Methods $\mathrm{M}_1$ and $\mathrm{M}_2$ are \emph{equivalent}, if there exists a bijection $B$ from the degrees of freedom of $\mathrm{M}_1$ to those of $\mathrm{M}_2$, such that upon the action of $B$ the update equations of the degrees of freedom of $\mathrm{M}_1$ become those of $\mathrm{M}_2$.
\end{definition}

In the following we show that the DG method, naturally associated to $V_\text{br}^K$, is equivalent to Active Flux, which is naturally associated to $V^{K+1}$. Often methods are equivalent because they consider different bases of the same approximation space (e.g. modal and nodal DG); in the present case the situation is more interesting since the approximation spaces are not the same. Also, a DG method with polynomial degree $K$ is usually associated to a $K+1$ order of accuracy; Active Flux with polynomial degree $K+1$ has order of accuracy $K+2$ (\cite{abgrall22}). The equivalence between the methods thus is a path towards understanding the superconvergence of DG. 

This paper has been greatly influenced by ideas in \cite{roe17superconvergence} and \cite{cockburn25}, one of which, however, lacks a general proof, and the other does not make a link to Active Flux. In this paper we close these gaps and generalize the results significantly.

\subsection{One-dimensional linear problems}

The semi-discrete Active Flux method \eqref{eq:af2}--\eqref{eq:af1} for the linear advection equation $\del_t q + U \del_x q = 0$
for $U > 0$ reads as follows:
\begin{align}
 \frac{\dd}{\dd t} Q_{i+\frac12} + U \frac{\dd}{\dd x} Q_i\Big|_{x = \frac{\Delta x}{2}} &= 0 \label{eq:af1linear}\\
 \frac{\dd}{\dd t} Q_i^{(k)} - U \frac{A_k}{\Delta x} \int_{-\frac{\Delta x }{2}}^{\frac{\Delta x}{2}} b_k'(x) Q_i(x) \,\dd x+ U A_k\frac{b_k^+ Q_{i+\frac12} - b_k^- Q_{i-\frac12}}{\Delta x}   &=0 \label{eq:af2linear}
\end{align}
For $U < 0$, the point value is upwinded the other way:
\begin{align}
 \frac{\dd}{\dd t} Q_{i+\frac12} + U \frac{\dd}{\dd x} Q_{i+1}\Big|_{x = -\frac{\Delta x}{2}} &= 0 \label{eq:af1negative}
 \end{align}
In \cite{barsukow25afasfe,barsukow2025sbp} also the average of these updates has been considered (central Active Flux):
\begin{align}
 \frac{\dd}{\dd t} Q_{i+\frac12} + \frac12 U \left(\frac{\dd}{\dd x} Q_i\Big|_{x = \frac{\Delta x}{2}} + \frac{\dd}{\dd x} Q_{i+1}\Big|_{x = -\frac{\Delta x}{2}} \right )&= 0 \label{eq:af1central}
 \end{align}

The approximation $Q_i$ in a cell can be written in a basis dual to the degrees of freedom
\begin{align}
 Q_i(x) = Q_{i-\frac12} R^{K+1}_\text{L}(x) + \sum_{k=0}^{K-1} Q_i^{(k)} S^{K+1}_k(x) + Q_{i+\frac12} R^{K+1}_\text{R}(x) \qquad x \in \left[ -\frac{\Delta x}{2}, \frac{\Delta x}{2} \right] \label{eq:af1drecon}
\end{align}
with basis functions $R^{K+1}_\text{L/R}, S^{K+1}_k \in P^{K+1}$. In fact, $R^{K+1}_\text{L/R}$, associated to the point values, are Radau polynomials with the choice $I = \left[-\frac{\Delta x}{2}, \frac{\Delta x}{2}\right]$ in \eqref{eq:radaudefint} and with the normalization choice is as follows:
\begin{align}
 R^{K+1}_\text{L}\left(-\frac{\Delta x}{2} \right) &= 1 &R^{K+1}_\text{R}\left(-\frac{\Delta x}{2} \right) &= 0 &
 R^{K+1}_\text{L}\left(\frac{\Delta x}{2} \right) &= 0 & R^{K+1}_\text{R}\left(\frac{\Delta x}{2} \right) &= 1 \label{eq:radauproperties}
 \end{align}
The analogue of Equation \eqref{eq:radaudefint} for them reads
 \begin{align}
 \int_{-\frac{\Delta x}{2}}^{\frac{\Delta x}{2}} b(x) R^{K+1}_\text{L}(x) \, \dd x &= 0 \qquad \forall b \in P^{K-1}
\end{align}
Radau polynomials are thus a most natural ingredient of the approximation space of Active Flux.

Next, the equivalence between DG and AF for 1-d linear advection is shown. The DG method for the linear advection equation is given by the variational form
\begin{align}
 \frac{1}{\Delta x} \int_{-\frac{\Delta x}{2}}^{\frac{\Delta x}{2}} v \frac{\dd}{\dd t} q_i \, \dd x -U \frac{1}{\Delta x} \int_{-\frac{\Delta x}{2}}^{\frac{\Delta x}{2}}  v'  q_i\, \dd x + U \frac{v^+ \hat q_{i+\frac12} - v^- \hat q_{i-\frac12}}{\Delta x} = 0 \label{eq:dglinear}
\end{align}
for all $v \in P^K\left(\left[-\frac{\Delta x}{2}, \frac{\Delta x}{2}\right]\right)$. Here, $\hat q_{i+\frac12}$ is obtained from the numerical flux $\hat f_{i+\frac12} \simeq Uq$ as $\hat q_{i+\frac12} := \hat f_{i+\frac12} / U$ ($U\neq 0$).
The upwind flux amounts to the choices
\begin{align}
 \hat q_{i+\frac12} = \begin{cases} q_i^+ & U \geq 0 \\ q_{i+1}^- & U < 0 \end{cases} \label{eq:upwindfluxlinear}
\end{align}

We generalize \cite{cockburn25}, and adapt the proof to Active Flux. First, upon an integration by parts, \eqref{eq:dglinear} can be rewritten as follows:
\begin{align}
 \frac{1}{\Delta x} \int_{-\frac{\Delta x}{2}}^{\frac{\Delta x}{2}} v \frac{\dd}{\dd t} q_i \, \dd x + U \frac{1}{\Delta x} \int_{-\frac{\Delta x}{2}}^{\frac{\Delta x}{2}}  v  q_i'\, \dd x + U \frac{v^+ (\hat  q_{i+\frac12} - q_i^+ ) - v^- (\hat q_{i-\frac12} - q_i^- )}{\Delta x} = 0
\end{align}
\begin{figure}
 \centering
 \includegraphics[width=0.4\textwidth]{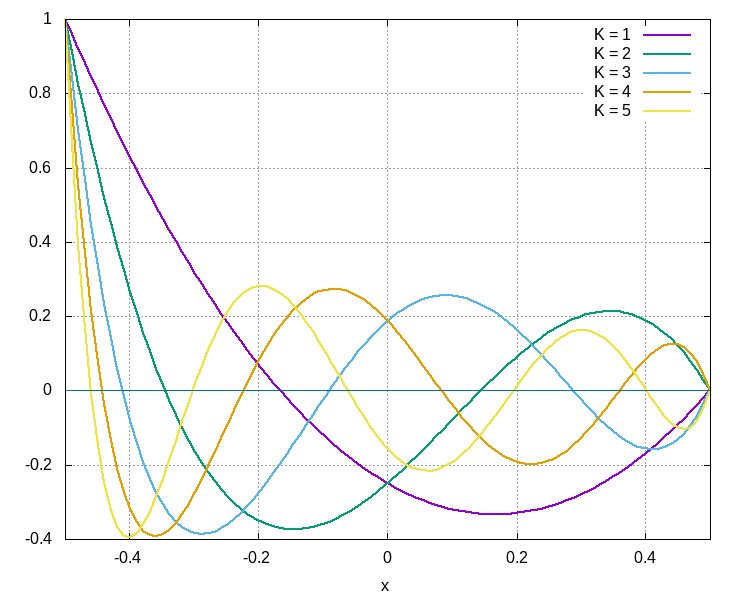} \hfill
 \includegraphics[width=0.49\textwidth]{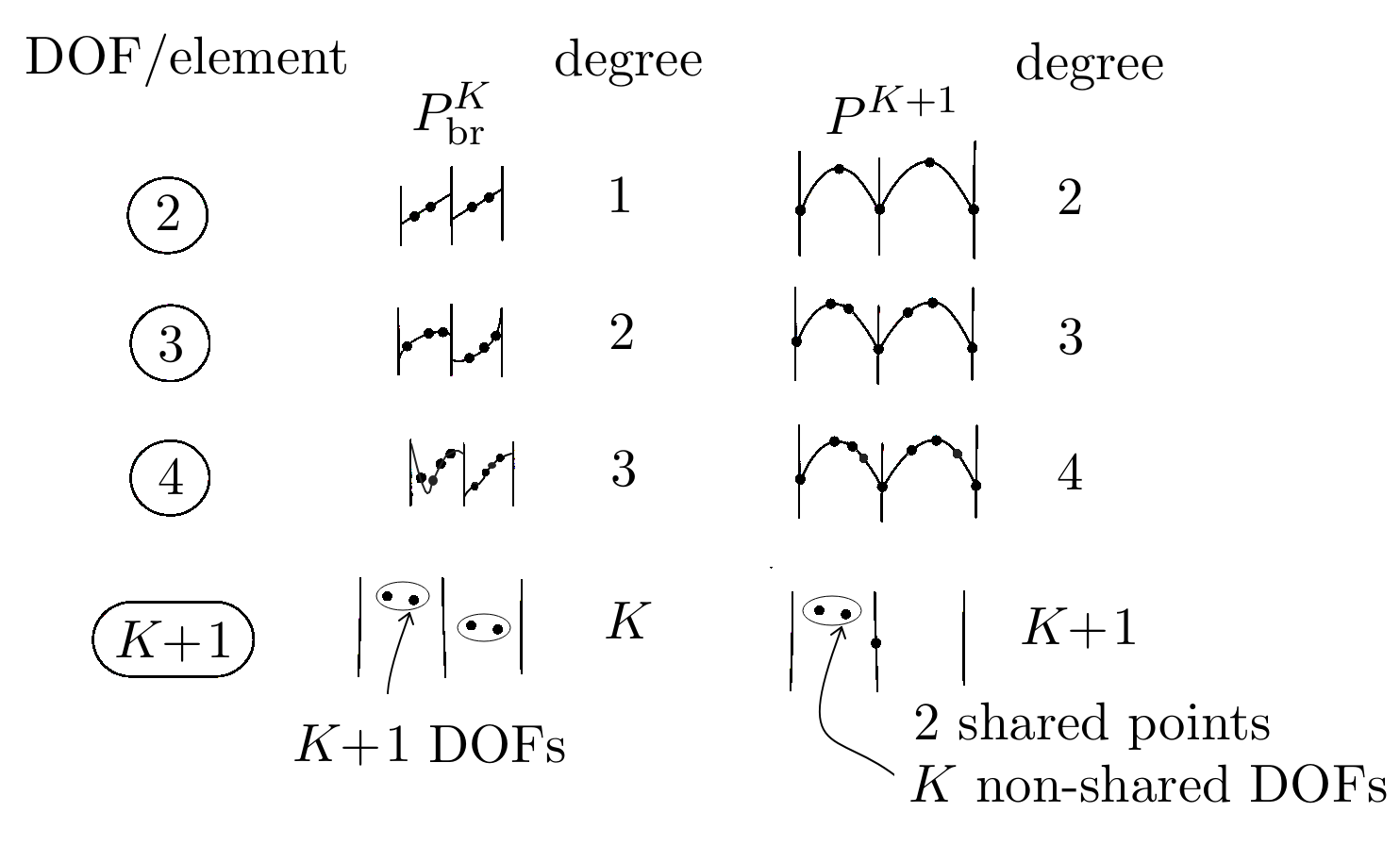}
 \caption{\emph{Left}: Radau polynomials $R^{K+1}_\text{L}$ for values $K \in \{ 1,\dots,5 \}$, with $\Delta x = 1$. \emph{Right}: Overview of the equivalent DG and AF methods and their degrees of freedom.}
 \label{fig:overview}
\end{figure}
Observe that now
\begin{align}
 -\frac{v^-}{\Delta x} &= \frac{1}{\Delta x} \int_{-\frac{\Delta x}{2}}^{\frac{\Delta x}{2}} \frac{\dd}{\dd x} (v R^{K+1}_\text{L}) \,\dd x = \frac{1}{\Delta x} \int_{-\frac{\Delta x}{2}}^{\frac{\Delta x}{2}}  v \frac{\dd}{\dd x}  R^{K+1}_\text{L} \,\dd x
\end{align}
and similarly for $v^+$ and $R_\text{R}$. The latter equality holds true because $v' \in P^{K-1}$. 
The DG method thus can be written as

\begin{align}
 \frac{1}{\Delta x} \int_{-\frac{\Delta x}{2}}^{\frac{\Delta x}{2}} v \frac{\dd}{\dd t} q_i \, \dd x + U \frac{1}{\Delta x} \int_{-\frac{\Delta x}{2}}^{\frac{\Delta x}{2}}  v  \frac{\dd}{\dd x}\Big( q_i  + (\hat q_{i+\frac12} - q_i^+ ) R^{K+1}_\text{R} + (\hat q_{i-\frac12} - q_i^- ) R^{K+1}_\text{L}\Big) \,\dd x = 0 \label{eq:dgfinal}
\end{align}

\begin{remark}\label{rem:backindgspace}
Observe that the derivative in the second integral is a polynomial of degree $K$.
\end{remark}

\begin{ident} \label{id:1d}
 Given the DG approximation, the mapping to the degrees of freedom of Active Flux is defined as
 \begin{align}
  Q_{i+\frac12} &:= \hat q_{i+\frac12} 
  &
  Q_i^{(k)} &:= \frac{A_k}{\Delta x} \int_{-\frac{\Delta x}{2}}^{\frac{\Delta x}{2}} b_k q_i \, \dd x \qquad \forall k \leq K-1
 \end{align}
\end{ident}
\begin{remark} 
Consider the upwind numerical flux \eqref{eq:upwindflux} for $U > 0$. Then, $Q_{i+\frac12} := q_i^+$, but $q_i^- \neq Q_{i-\frac12} = q_{i-1}^+$.
\end{remark}

Now, more can be said about the function appearing under the derivative in \eqref{eq:dgfinal}:
\begin{theorem}
 The (broken) DG approximation $q_i \in P^K$ augmented by the two Radau polynomials 
 \begin{align}
  q_i  + (\hat q_{i+\frac12} - q_i^+ ) R^{K+1}_\text{R} + (\hat q_{i-\frac12} - q_i^- ) R^{K+1}_\text{L} 
 \end{align}
 is the (globally continuous) AF approximation $Q_i\in P^{K+1}$ as it was introduced in Section \ref{ssec:af1d}.
\end{theorem}
\begin{proof}
 Using the defining properties \eqref{eq:radauproperties} of the Radau polynomials and Identification \ref{id:1d} one finds
 \begin{align*}
  Q_i\left( \frac{\Delta x}{2} \right ) &= 
  %q_i^+  + (\hat q_{i+\frac12} - q_i^+ ) = 
  Q_{i+\frac12} &
  Q_i\left( -\frac{\Delta x}{2} \right ) &= 
  %q_i^-  + (\hat q_{i-\frac12} - q_i^- ) = 
  Q_{i-\frac12}  &
  \frac{A_k}{\Delta x} \int_{-\frac{\Delta x}{2}}^{\frac{\Delta x}{2}} b_k Q_i \, \dd x &=  Q_i^{(k)}
 \end{align*}
 Thus the DG approximation plus the scaled Radau polynomials is the (continuous) Active Flux approximation (of one polynomial degree higher).
\end{proof}

The final question is whether the DG method implies the same update equations for the degrees of freedom of Active Flux as the classical Active Flux method. It is answered in the affirmative by the following

\begin{theorem}
 For $K\geq 1$, consider the DG method \eqref{eq:dg} with a linear two-point numerical flux $\hat f$, $\hat f \colon (q_\text{L}, q_\text{R}) \mapsto \hat f(q_\text{L}, q_\text{R})$, such that $\hat f_{i+\frac12} = \hat f(q_i^+, q_{i+1}^-)$ is a linear combination $U (\alpha^+ q_i^+ + \alpha^- q_{i+1}^-)$ of the two states. Then, this is the Active Flux method for the degrees of freedom given in the Identification \ref{id:1d}, with the average update \eqref{eq:af1} and the point value update employing the same linear combination of the two adjacent slopes, i.e.
 \begin{align}
 \frac{\dd}{\dd t} Q_{i+\frac12} + U \left(\alpha^+ \frac{\dd}{\dd x} Q_i\Big|_{x = \frac{\Delta x}{2}} + \alpha^- \frac{\dd}{\dd x} Q_{i+1}\Big|_{x = -\frac{\Delta x}{2}} \right )&= 0 \label{eq:lincombpointupdate1d}
 \end{align}
\end{theorem}
\begin{proof}
Use the form \eqref{eq:dgfinal} of the DG method: By Riesz' representation theorem, there exist $v_\text{L}, v_\text{R} \in P^K$ such that
 \begin{align}
  \frac{1}{\Delta x} \int_{-\frac{\Delta x}{2}}^{\frac{\Delta x}{2}} v_\text{L} q_i \, \dd x  &= q_i^- \label{eq:leftprojector}\\
  \frac{1}{\Delta x} \int_{-\frac{\Delta x}{2}}^{\frac{\Delta x}{2}} v_\text{R} q_i \, \dd x  &= q_i^+ \label{eq:rightprojector}
 \end{align}
 By linearly combining \eqref{eq:dgfinal} for cells $i$ and $i+1$ with, respectively, $v_\text{R}$ and $v_\text{L}$ as test function, one obtains the evolution equation for $Q_{i+\frac12}$ through Identification \ref{id:1d}. Due to Remark \ref{rem:backindgspace}, the second integral in \eqref{eq:dgfinal} will then be reducing, respectively, to the derivative of $Q_i$ at $x = \frac{\Delta x}{2}$ and that of $Q_{i+1}$ at $x = -\frac{\Delta x}{2}$, giving in total rise to the linear combination given in \eqref{eq:lincombpointupdate1d}.
 For the choice of the upwind flux this reduces to Equations \eqref{eq:af1linear}/\eqref{eq:af1negative} and to \eqref{eq:af1central} for the central flux. The update equation \eqref{eq:af2} for the moments is immediately obvious by choosing $v = b_k$.
\end{proof}
The approach immediately generalizes to linear systems in 1-d (see \eqref{eq:af1}).

\begin{corollary}
 Consider the upwind flux \eqref{eq:upwindflux} and, for definiteness, $U > 0$. Then, the Active Flux approximation $Q_i$ is
 \begin{align}
  Q_i(x) = q_i(x) + ( \hat q_{i-\frac12} - q_i^-) R^{K+1}_\text{L}(x) \label{eq:postprocessed}
 \end{align}
 Active Flux is a method of order of accuracy $K+2$, while DG is usually associated with an order of accuracy $K+1$. However, at the zeros of $R^{K+1}_\text{L}$ (so-called \emph{downwind Radau points}, e.g. \cite{yang12superconvergence}) both approximations agree and the higher order of accuracy becomes visible as superconvergence of DG. We cannot at this point establish the $2K+1$ superconvergence, valid only at the cell interface and not in the interior Radau points.
\end{corollary}

\begin{remark}
 For the central flux $\hat q_{i+\frac12} = \frac{q_i^+ + q_{i+1}^-}{2}$, the difference between the DG and the AF approximations remains a weighted sum of both Radau polynomials. Their zeros are distinct such that none of the two sets of points allows to measure superconvergence directly. However, one can study the zeros of the weighted sum of Radau polynomials, see \cite{frean20}. The superconvergence would also become measurable in the moments (which agree), or upon post-processing.
\end{remark}
\begin{remark}
 The important Remark \ref{rem:backindgspace} can be rephrased as: The derivative of the AF approximation is in the (broken) DG space. One will see below that the generalization of this statement to multiple dimensions is not straight forward.
\end{remark}

\subsection{One-dimensional nonlinear problems}

While the equivalence is very clear for linear problems, it is less so for nonlinear ones, since there is no unique AF method for them. 
The point value update requires upwinding, and there are various ways of splitting the Jacobian, and various ways of approximating the flux derivative (compare Jacobian splitting in \cite{barsukow24afeuler} and Flux-Vector Splitting (FVS) of \cite{duan24}). We show in the following that in the nonlinear case DG can be rephrased as a new Active Flux method somewhere between the two choices cited.

To this end, consider the $m \times m$ hyperbolic system \eqref{eq:conslaw} of conservation laws in this Section.
Choosing $v \equiv 1$ in \eqref{eq:dg} yields the update of the average of the DG approximation
\begin{align}
 \frac{\dd}{\dd t } \frac{1}{\Delta x} \int_{-\frac{\Delta x}{2}}^{\frac{\Delta x}{2}} q_i \, \dd x + \frac{\hat f_{i+\frac12} - \hat f_{i-\frac12}}{\Delta x} = 0 
\end{align}
The average update (see \eqref{eq:af2}) for AF reads
\begin{align}
 \frac{\dd}{\dd t} Q_i^{(0)} + \frac{f(Q_{i+\frac12}) - f(Q_{i-\frac12})}{\Delta x} = 0
\end{align}
such that it is natural to identify the two averages and $Q_{i+\frac12} := f^{-1}(\hat f_{i+\frac12})$ wherever $f$ is invertible. (In our opinion this identification really justifies the name ``Active Flux''.) 
For linear $f$, it reduces to Identification \ref{id:1d} discussed previously. 
Then, the update equation for the point value must be
\begin{align}
 \frac{\dd}{\dd t} Q_{i+\frac12} = \Big(f'(f^{-1}(\hat f_{i+\frac12})) \Big)^{-1} \left( \frac{\del \hat f}{\del q_\text{L}} \frac{\dd}{\dd t} q_i^+ +\frac{\del \hat f}{\del q_\text{R}} \frac{\dd}{\dd t} q_{i+1}^-  \right)
\end{align}
with a two-point numerical flux $\hat f(q_\text{L}, q_\text{R})$, as defined in \eqref{eq:twopointflux}.

Recall the test functions $v_\text{L/R}$ defined in \eqref{eq:leftprojector}--\eqref{eq:rightprojector}, which project out $q_i^\pm$. Define also $A_{i+\frac12} := f'(f^{-1}(\hat f_{i+\frac12}))$, an (unusual) estimate of the Jacobian. Since $f \circ q$ is not in $P^{K+1}$ (or not even polynomial in general), one cannot proceed as before. Define thus a projection $F_i \in P^{K+1}$ of $f\circ q_i$ into the space of AF approximations such that
\begin{subequations}
\begin{align}
 \int_{-\frac{\Delta x}{2}}^{\frac{\Delta x}{2}} b (F_i - f \circ q_i) \, \dd x &= 0 \qquad \qquad \forall b \in P^{K-1} \\
 F_i^\pm := F_i\left(\pm \frac{\Delta x}{2}\right ) &= f(Q_{i\pm\frac12}) \equiv \hat f_{i\pm\frac12} 
\end{align} \label{eq:fluxprojection}
 \end{subequations}
Observe also that $F \neq f \circ Q_i$ in general.
Then, the DG method becomes (no approximation here!), $\forall v \in P^K$
\begin{align}
 \frac{1}{\Delta x} \int_{-\frac{\Delta x}{2}}^{\frac{\Delta x}{2}} v \frac{\dd}{\dd t }q_i \, \dd x  - \frac{1}{\Delta x} \int_{-\frac{\Delta x}{2}}^{\frac{\Delta x}{2}} v' F_i \, \dd x + \frac{v^+ F_i^+ - v^- F_i^-}{\Delta x} &= 0 \\
 \frac{1}{\Delta x} \int_{-\frac{\Delta x}{2}}^{\frac{\Delta x}{2}} v \frac{\dd}{\dd t }q_i \, \dd x  + \frac{1}{\Delta x} \int_{-\frac{\Delta x}{2}}^{\frac{\Delta x}{2}} v  \frac{\dd}{\dd x} F_i \, \dd x  &= 0
\end{align}
which is the equivalent of \eqref{eq:dgfinal}. Since $\frac{\dd}{\dd x} F_i \in P^{K}$, one concludes, as before
\begin{align}
 \frac{\dd}{\dd t} Q_{i+\frac12} &= A^{-1}_{i+\frac12} \left( \frac{\del \hat f}{\del q_\text{L}} 
 \frac{\dd}{\dd x} F_i  \Big |_{x = \frac{\Delta x}{2}} + 
 \frac{\del \hat f}{\del q_\text{R}} 
 \frac{\dd}{\dd x} F_{i+1}  \Big |_{x = -\frac{\Delta x}{2}}
 \right)
\end{align}
Of course, the derivatives of $F_i$ are different from those of $f\circ q_i$, which is where the projection \eqref{eq:fluxprojection} becomes relevant.

\begin{example}
For $K=1$ one finds, with $\bar f_i := \frac{1}{\Delta x} \int_{-\frac{\Delta x}{2}}^{\frac{\Delta x}{2}} f(q_i) \, \dd x $
 \begin{align*}
 \frac{\dd}{\dd t} Q_{i+\frac12} &= A^{-1}_{i+\frac12} \left( \frac{\del \hat f}{\del q_\text{L}} 
 \left( \frac{6\bar f_i  - 4 f(Q_{i+\frac12})-2 f(Q_{i-\frac12})}{\Delta x}  \right )
 %\right . \\\nonumber & \qquad \qquad \left. 
 +\frac{\del \hat f}{\del q_\text{R}} 
 \left(  \frac{-6 \bar f_{i+1}  +2 f(Q_{i+\frac32}) + 4 f(Q_{i+\frac12})}{\Delta x} \right) 
 \right)
\end{align*} 
\end{example}

Since the same Finite-Difference-like formulas are used as before, but applied to $F$ instead of $Q$ (as in \eqref{eq:af1}), it is worth introducing the notation
\begin{align}
 (DQ_i)^+ &:= \frac{\dd}{\dd x} Q_i \Big |_{x = \frac{\Delta x}{2}} & 
 (DQ_{i+1})^- &:= \frac{\dd}{\dd x} Q_{i+1} \Big |_{x = -\frac{\Delta x}{2}}
\end{align}

\begin{example}
Consider as $\hat f$ the global Lax-Friedrichs' flux, such that
\begin{align}
 \hat f(q_\text{L} , q_\text{R}) = \frac{f(q_\text{L}) + f(q_\text{R})}{2} - \frac{a}{2} (q_\text{R} - q_\text{L})
\end{align}
for some $a = \mathrm{const}$. Then
\begin{align}
 \frac{\del \hat f}{\del q_\text{L/R}} = \frac12 (f'(q_\text{L/R}) \mp a))
\end{align}
Thus, $\frac{\del \hat f}{\del q_\text{L/R}}$ are a particular way of choosing the positive/negative part $J^\pm_{i+\frac12}$ of the Jacobian $f'$ at $x_{i+\frac12}$. With this, the above point value update can be written as
\begin{align}
 \frac{\dd}{\dd t} Q_{i+\frac12} &=- \left(A^{-1}_{i+\frac12} J^+_{i+\frac12} (DF_i)^+ + A^{-1}_{i+\frac12} J^-_{i+\frac12} (DF_{i+1})^-\right)
\end{align}

Compare this to the approach
\begin{align}
 \frac{\dd}{\dd t} Q_{i+\frac12} &=- \Big(J^+_{i+\frac12} (DQ_i)^+  + J^-_{i+\frac12} (DQ_{i+1})^-\Big)
\end{align}
from \cite{barsukow24afeuler} (where $J^\pm$ were instead defined via the eigenvalues of $J$) and the one from \cite{duan24}:
\begin{align}
 \frac{\dd}{\dd t} Q_{i+\frac12} &=- \Big( (Df_i^+)^+ + (Df_{i+1}^-)^-\Big)
\end{align}
with the Flux-Vector Splitting $f_i^\pm = \frac12 (f(Q_i) \pm a Q_i)$ (for example). The DG-inspired AF thus is closely related to existing variants of AF for nonlinear problems. 
\end{example}

Most importantly, one thus concludes that DG can be written as an Active Flux method even for nonlinear problems.

\subsection{Multi-dimensional linear problems on Cartesian grids}

\subsubsection{A special choice of degrees of freedom for AF}

Above, it has been shown that in 1-d, the classical DG method for linear problems becomes the classical AF method upon a suitable mapping of the degrees of freedom: the two methods are equivalent in the sense of Definition \ref{def:equiv}. A natural DG method on 2-d Cartesian grids employs tensor-product basis and test functions. 

It turns out (already for third-order accuracy) that it is \emph{not} equivalent to the \emph{classical} AF method employing point values at edge midpoints, point values at nodes and cell averages. However, equivalence can be shown for a non-classical point value update. Here, we find it more instructive to instead consider a new, non-classical choice of degrees of freedom for Active Flux, but endowed with a natural update, in a way made precise below. The new third-order accurate, \emph{tensorial} Active Flux method (see Figure \ref{fig:newAF}) uses as degrees of freedom point values at nodes (as before), a cell average (also unchanged) and one-dimensional averages along the edges (instead of values at their midpoints):
\begin{align}
 \bar Q_{i+\frac12,j} \simeq \frac{1}{\Delta y} \int_{-\frac{\Delta y}{2}}^{\frac{\Delta y}{2}} q(t, x, y) \, \dd y \\
 \bar Q_{i,j+\frac12} \simeq \frac{1}{\Delta x} \int_{-\frac{\Delta x}{2}}^{\frac{\Delta x}{2}} q(t, x, y) \, \dd x
\end{align}
These latter will be referred to as \emph{edge averages}.

These degrees of freedom are those of Virtual Finite Elements (VEM) (see e.g. \cite{beirao13}, and in particular \cite{abgrall24pampa} for their usage in Active Flux methods on triangular grids). These degrees of freedom might also call for comparison with those of the Whitney finite elements, which associate to each $k$-cell (in the language of algebraic topology) an average over it. Here, however, orientation is not taken into account and all variables remain scalars. One-dimensional averages were also employed recently in \cite{grunwald25}, albeit only as an approximation in conjunction with a splitting into one-dimensional problems, and not in an effort to explore alternative multi-dimensional formulations of Active Flux.

\begin{figure}
 \centering
 \includegraphics[width=0.6\textwidth]{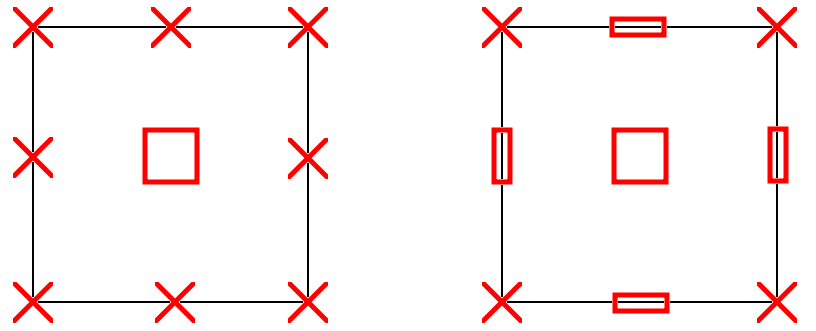}
 \caption{\emph{Left}: Classical degrees of freedom of Active Flux on Cartesian grids (third-order accuracy). \emph{Right}: The modified degrees of freedom proposed here. The edge midpoints are replaced by one-dimensional averages along the edge.}
 \label{fig:newAF}
\end{figure}

The special property associated with the usage of edge averages instead of midpoints is the tensor-product structure of the basis functions. If the one-dimensional AF approximation $Q_i$ is written as
\begin{align}
 Q_i(x) = Q_{i-\frac12} R_\text{L}^{K+1}(x) + Q_i^{(0)} S_0^{K+1}(x) + Q_{i+\frac12} R_\text{R}^{K+1}(x)
\end{align}
(see Equation \eqref{eq:af1drecon}), then the two-dimensional tensor-product approximation is
\begin{align}
 Q_{ij}(x) &= Q_{i-\frac12,j-\frac12} R_{\text{L},\Delta x}^{K+1}(x) R_{\text{L},\Delta y}^{K+1}(y) + Q_{i+\frac12,j-\frac12} R_{\text{R},\Delta x}^{K+1}(x) R_{\text{L},\Delta y}^{K+1}(y) + \dots    
 \\&\nonumber + Q_{ij}^{(0)} S_{0,\Delta x}^{K+1}(x)S_{0,\Delta y}^{K+1}(y) 
 + \bar Q_{i-\frac12,j} R_{\text{L},\Delta x}^{K+1}(x)S_{0,\Delta y}^{K+1}(y)+ \bar Q_{i+\frac12,j} R_{\text{R},\Delta x}^{K+1}(x)S_{0,\Delta y}^{K+1}(y) \\
 &+ \nonumber \bar Q_{i,j-\frac12} S_{0,\Delta x}^{K+1}(x)R_{\text{L},\Delta y}^{K+1}(y)+ \bar Q_{i,j+\frac12} S_{0,\Delta x}^{K+1}(x)R_{\text{R},\Delta y}^{K+1}(y)
\end{align}
The usage of classical edge midpoints does not lead to a tensor-product structure of the basis function of AF, as pointed out in \cite{barsukow25afasfe}.
This property, however, is not used directly to show equivalence between DG and AF below.

Simpson's rule allows to convert between point values and edge averages without losing accuracy:
\begin{align}
 \bar Q_{i+\frac12,j} = \frac{Q_{i+\frac12,j+\frac12} + 4 Q_{i+\frac12,j} + Q_{i+\frac12,j-\frac12}}{6}
\end{align}
Why does it matter then whether a point value or a one-dimensional average is used?
The difference lies in what one considers their natural update. 

\begin{example}
Consider linear advection with $U^x, U^y > 0$
\begin{align}
 \del_t q + U^x \del_x q+ U^y \del_y q = 0 \qquad U^x, U^y \in \mathbb R^+ \label{eq:linadvmultid}
\end{align}
The edge midpoint value is naturally updated as
\begin{align}
 \frac{\dd}{\dd t} Q_{i+\frac12,j} + U^x \frac{\del}{\del x} Q_{ij} \Big|_{x = \frac{\Delta x}{2}, y = 0} + U^y \frac{\del}{\del y} Q_{ij} \Big|_{x = \frac{\Delta x}{2}, y = 0} = 0 \label{eq:updateedgemidpoint}
\end{align}
while the edge-average is naturally updated by integrating \eqref{eq:linadvmultid} over the edge and thus as
\begin{align}
 \frac{\dd}{\dd t} \bar Q_{i+\frac12,j} + U^x  \frac{1}{\Delta y} \int_{-\frac{\Delta y}{2}}^{\frac{\Delta y}{2}} \frac{\del}{\del x} Q_{ij} \Big|_{x = \frac{\Delta x}{2}} \dd y + U^y \frac{Q_{i+\frac12, j+\frac12} - Q_{i+\frac12,j-\frac12}}{\Delta y} = 0 \label{eq:updateedgeaverage}
\end{align}

The integral would be replaced by quadrature as:
\begin{align}
 \frac{1}{\Delta y} \int_{-\frac{\Delta y}{2}}^{\frac{\Delta y}{2}} \frac{\del}{\del x} Q_{ij} \Big|_{x = \frac{\Delta x}{2}} \dd y = \frac16 \left( \frac{\del}{\del x} Q_{ij} \Big|_{x = \frac{\Delta x}{2}, y=\frac{\Delta y}{2}} + 4 \frac{\del}{\del x} Q_{ij} \Big|_{x = \frac{\Delta x}{2}, y=0} + \frac{\del}{\del x} Q_{ij} \Big|_{x = \frac{\Delta x}{2}, y=-\frac{\Delta y}{2}}    \right )
\end{align}
Explicit calculation (omitted) shows that \eqref{eq:updateedgeaverage} and \eqref{eq:updateedgemidpoint} are not equivalent, in the sense that
\begin{align}
 \frac{\dd}{\dd t} \bar Q_{i+\frac12,j}  \neq \frac16 \left( \frac{\dd}{\dd t} Q_{i+\frac12,j+\frac12}  + 4 \frac{\dd}{\dd t} Q_{i+\frac12,j}+ \frac{\dd}{\dd t} Q_{i+\frac12,j-\frac12}  \right) \label{eq:updateedgemismatch}
\end{align}
\end{example}

One can, however, use \eqref{eq:updateedgemismatch} to \emph{define} a modified update of $Q_{i+\frac12,j}$. This way, one will end up with an AF method employing classical degrees of freedom, in particular, point values at edge midpoints, but with a new update prescription for the latter. We prefer the opposite choice here, and use edge averages as non-classical degrees of freedom, but endow them with the natural update \eqref{eq:updateedgeaverage}.

For completeness, the natural update of the nodal point values is (see e.g. \cite{barsukow24afeuler})
\begin{align}
 \frac{\dd}{\dd t} Q_{i+\frac12,j+\frac12} + U^x \frac{\del}{\del x} Q_{ij} \Big|_{x = \frac{\Delta x}{2}, y = \frac{\Delta y}{2}} + U^y \frac{\del}{\del y} Q_{ij} \Big|_{x = \frac{\Delta x}{2}, y = \frac{\Delta y}{2}} = 0 \label{eq:updatenodal}
\end{align}
and that of the average is
\begin{align}
 \frac{\dd}{\dd t} Q^{(0)}_{ij} + U^x \frac{\bar Q_{i+\frac12,j} - \bar Q_{i-\frac12,j}}{\Delta x} + U^y \frac{\bar Q_{i, j+\frac12} - \bar Q_{i,j-\frac12}}{\Delta y} = 0 \label{eq:updateaverage}
\end{align}

Conversely, it should also be possible to stick to a classical AF method, but a devise a non-classical DG method that is equivalent to it. We will address this as part of future work.

\subsubsection{The equivalence of DG and AF}

In this Section, we focus on showing equivalence between DG and tensorial AF in the multi-dimensional case for linear advection only, and postpone the nonlinear case and that of systems to future work due to not insignificant technical challenges associated with these.

Consider DG on Cartesian grids with a tensor basis, i.e. test and basis functions $v(x) w(y)$, $v \in P^K\left(\left [ - \frac{\Delta x}{2}, \frac{\Delta x}{2} \right ] \right )$, $w \in P^K\left(\left [ - \frac{\Delta y}{2}, \frac{\Delta y}{2} \right ] \right )$, for two-dimensional linear advection \eqref{eq:linadvmultid} with no restrictions on the signs of $U^x$ and $U^y$:
\begin{align}
 &\frac{1}{\Delta x} \int_{-\frac{\Delta x}{2}}^{\frac{\Delta x}{2}} \frac{1}{\Delta y} \int_{-\frac{\Delta y}{2}}^{\frac{\Delta y}{2}} v w \frac{\dd}{\dd t} q_{ij} \,  \dd y \dd x  - \frac{1}{\Delta x} \int_{-\frac{\Delta x}{2}}^{\frac{\Delta x}{2}} \frac{1}{\Delta y} \int_{-\frac{\Delta y}{2}}^{\frac{\Delta y}{2}} \left( v' w U^x q_{ij} + v w' U^y q_{ij}\right) \,  \dd y \dd x \label{eq:dg2d}  \\
 &+ \nonumber\frac{1}{\Delta y} \int_{-\frac{\Delta y}{2}}^{\frac{\Delta y}{2}} w U^x \frac{v^+ \hat q_{i+\frac12,j} - v^- \hat q_{i-\frac12,j}}{\Delta x } \dd y + \frac{1}{\Delta x} \int_{-\frac{\Delta x}{2}}^{\frac{\Delta x}{2}} v U^y \frac{w^+ \hat q_{i,j+\frac12} - w^- \hat q_{i,j-\frac12}}{\Delta y} \dd x
 = 0
\end{align}
Recall that $\hat f^x_{i+\frac12,j}$ is obtained by means of a numerical flux $\hat f^x \colon \mathbb R \times \mathbb R \to \mathbb R$ as
\begin{align}
 \hat f^x_{i+\frac12,j}(t, y) := \hat f^x\left(q_{ij}\left(t, \frac{\Delta x}{2}, y\right), q_{i+1,j}\left(t, -\frac{\Delta x}{2}, y\right)\right)
\end{align}
The time-dependence in $q_{ij}$ will be again omitted now, and analogously for $\hat f^y$.

Here, $\hat q_{i+\frac12,j},\hat q_{i,j+\frac12} \colon \mathbb R^+_0 \times \mathbb R \to \mathbb R$ are obtained from the numerical fluxes $\hat f^x_{i+\frac12,j} \simeq U^x q$, $\hat f^y_{i,j+\frac12} \simeq U^y q$ as
\begin{align}
 \hat q_{i+\frac12,j}  &:= \begin{cases} \hat f^x_{i+\frac12,j} / U^x & U^x \neq 0 \\ 0 & \text{else} \end{cases} & \hat q_{i,j+\frac12}  &:= \begin{cases} \hat f^y_{i,j+\frac12} / U^y & U^y \neq 0 \\ 0 & \text{else} \end{cases}
\end{align}
 
The upwind flux amounts to choosing
\begin{align}
  \hat q_{i+\frac12,j} &= \begin{cases} q_{ij}\left(\frac{\Delta x}{2}, y\right) & U^x > 0 \\ q_{i+1,j}\left(-\frac{\Delta x}{2}, y\right) & \text{else} \end{cases} & \hat q_{i,j+\frac12} &= \begin{cases} q_{ij}(x, \frac{\Delta y}{2})& U^y > 0 \\ q_{i,j+1}(x, -\frac{\Delta y}{2}) & \text{else} \end{cases}
\end{align}

Upon an integration by parts \eqref{eq:dg2d} becomes 

 \begin{align}
 \frac{1}{\Delta x} &\int_{-\frac{\Delta x}{2}}^{\frac{\Delta x}{2}} \frac{1}{\Delta y} \int_{-\frac{\Delta y}{2}}^{\frac{\Delta y}{2}} v w \frac{\dd}{\dd t} q_{ij} \,  \dd y \dd x  + \frac{1}{\Delta x} \int_{-\frac{\Delta x}{2}}^{\frac{\Delta x}{2}} \frac{1}{\Delta y} \int_{-\frac{\Delta y}{2}}^{\frac{\Delta y}{2}}  v w \left(U^x \del_x q_{ij} + U^y \del_y q_{ij}\right) \,  \dd y \dd x  \label{eq:multiddginterm} \\
  &+ \nonumber
 \frac{1}{\Delta y} \int_{-\frac{\Delta y}{2}}^{\frac{\Delta y}{2}} w U^x v^+\frac{ \hat q_{i+\frac12,j} - q_{ij}(\frac{\Delta x}{2}, y)  }{\Delta x } \dd y - 
 \frac{1}{\Delta y} \int_{-\frac{\Delta y}{2}}^{\frac{\Delta y}{2}} w U^x v^-\frac{ \hat q_{i-\frac12,j} -q_{ij}(-\frac{\Delta x}{2}, y)  }{\Delta x } \dd y \\ & \nonumber +
  \frac{1}{\Delta x} \int_{-\frac{\Delta x}{2}}^{\frac{\Delta x}{2}} v U^y w^+\frac{ \hat q_{i,j+\frac12} - q_{ij}(x, \frac{\Delta y}{2})  }{\Delta y} \dd x 
 -\frac{1}{\Delta x} \int_{-\frac{\Delta x}{2}}^{\frac{\Delta x}{2}} v U^y w^-\frac{ \hat q_{i,j-\frac12} - q_{ij}(x, -\frac{\Delta y}{2})}{\Delta y} \dd x 
 = 0 
\end{align}

This makes it natural to define
\begin{align}
 r^x_{ij}(x, y) &:= \left( \hat q_{i+\frac12,j} - q_{ij}\left(\frac{\Delta x}{2}, y\right) \right ) R_{\text{R},\Delta x}^{K+1}(x) + \left(\hat q_{i-\frac12,j} - q_{ij}\left(-\frac{\Delta x}{2}, y\right)\right ) R_{\text{L},\Delta x}^{K+1}(x) \label{eq:rxdef} \\
 r^y_{ij}(x, y) &:= \left( \hat q_{i,j+\frac12} - q_{ij}\left(x, \frac{\Delta y}{2}\right)  \right ) R_{\text{R}, \Delta y}^{K+1}(y) + \left(\hat q_{i,j-\frac12} - q_{ij}\left(x, -\frac{\Delta y}{2}\right)\right ) R_{\text{L},\Delta y}^{K+1}(y) \label{eq:rydef}
\end{align}
Here the length of the interval involved in the definition of the Radau polynomial is made explicit as an additional subscript.
Since the test functions factorize, one immediately has
\begin{align}
 \frac{1}{\Delta x} \int_{-\frac{\Delta x}{2}}^{\frac{\Delta x}{2}} \frac{1}{\Delta y} \int_{-\frac{\Delta y}{2}}^{\frac{\Delta y}{2}} v' w r^x_{ij} \, \dd y \dd x &= 0 &
 \frac{1}{\Delta x} \int_{-\frac{\Delta x}{2}}^{\frac{\Delta x}{2}} \frac{1}{\Delta y} \int_{-\frac{\Delta y}{2}}^{\frac{\Delta y}{2}} v w' r^y_{ij} \, \dd y \dd x &= 0
\end{align}
since
\begin{align}
 \int_{-\frac{\Delta x}{2}}^{\frac{\Delta x}{2}} b R_\text{L}^{K+1} \, \dd x = \int_{-\frac{\Delta x}{2}}^{\frac{\Delta x}{2}} b R_\text{R}^{K+1} \, \dd x  &= 0 \qquad \forall b \in P^{K-1}
\end{align}

By using the properties \eqref{eq:radauproperties} of Radau polynomials one thus finds
\begin{align}
 &\frac{1}{\Delta x} \int_{-\frac{\Delta x}{2}}^{\frac{\Delta x}{2}} \frac{1}{\Delta y} \int_{-\frac{\Delta y}{2}}^{\frac{\Delta y}{2}} v w \frac{\del}{\del x} r^x_{ij} \, \dd x \dd y =\frac{1}{\Delta x} \int_{-\frac{\Delta x}{2}}^{\frac{\Delta x}{2}} \frac{1}{\Delta y} \int_{-\frac{\Delta y}{2}}^{\frac{\Delta y}{2}} w \frac{\del}{\del x}(v  r^x_{ij}) \, \dd x \dd y\\
&=\frac{1}{\Delta y} \int_{-\frac{\Delta y}{2}}^{\frac{\Delta y}{2}} w(y) \left( v^+ \frac{ \hat q_{i+\frac12,j} - q_{ij}\left(\frac{\Delta x}{2}, y\right) }{\Delta x}   - v^- \frac{\hat q_{i-\frac12,j} - q_{ij}\left(-\frac{\Delta x}{2}, y\right)}{\Delta x}  \right) \dd y
\end{align}
and analogously for the other direction.
This can be used to replace the corresponding terms in \eqref{eq:multiddginterm} thus obtaining
\begin{align}
 \frac{1}{\Delta x} \int_{-\frac{\Delta x}{2}}^{\frac{\Delta x}{2}} \frac{1}{\Delta y} &\int_{-\frac{\Delta y}{2}}^{\frac{\Delta y}{2}} v w \frac{\dd}{\dd t} q_{ij} \,  \dd y \dd x  + \frac{1}{\Delta x} \int_{-\frac{\Delta x}{2}}^{\frac{\Delta x}{2}} \frac{1}{\Delta y} \int_{-\frac{\Delta y}{2}}^{\frac{\Delta y}{2}}  v w \left(U^x \del_x (q_{ij} + r^x_{ij}) + U^y \del_y (q_{ij} + r^y_{ij})\right) \,  \dd y \dd x  = 0 \label{eq:vwupdateintermiedate}
\end{align}

\begin{remark} \label{rem:afdgspace}
 The situation at this point differs from that of the one-dimensional case (which was Equation \eqref{eq:dgfinal}): Observe that the derivatives in \eqref{eq:vwupdateintermiedate} act on \emph{different} functions, and neither $q_{ij} + r^x_{ij}$, nor $q_{ij} + r^y_{ij}$ equals the AF approximation (as will be shown next). This is actually very good news: while both $\del_x (q_{ij} + r^x_{ij})$ and $\del_y (q_{ij} + r^y_{ij})$ are in the DG space $Q^K$, the same first derivatives of the $Q^{K+1}$ AF approximation would be outside of this space. Having the AF approximation appear there would not allow to generalize the 1-d approach.
\end{remark}

\begin{lemma} \label{lem:recon}
There exist $C_{ij}^\text{LL},C_{ij}^\text{LR},C_{ij}^\text{RL},C_{ij}^\text{RR} \in \mathbb R$ such that 
\begin{align}
 Q_{ij} &=  q_{ij} + r^x_{ij} + r^y_{ij} \\&\nonumber + C_{ij}^\text{LL} R_{\text{L},\Delta x}^{K+1}(x) R_{\text{L},\Delta y}^{K+1}(y) +
 C_{ij}^\text{LR} R_{\text{L},\Delta x}^{K+1}(x) R_{\text{R},\Delta y}^{K+1}(y)
 \\&\nonumber +C_{ij}^\text{RL} R_{\text{R},\Delta x}^{K+1}(x) R_{\text{L},\Delta y}^{K+1}(y)
 +C_{ij}^\text{RR} R_{\text{R},\Delta x}^{K+1}(x) R_{\text{R},\Delta y}^{K+1}(y) 
\end{align}
with $r^x_{ij}, r^y_{ij}$ given in \eqref{eq:rxdef}--\eqref{eq:rydef}
is the AF approximation.
\end{lemma}
The proof can be found in Section \ref{sec:lemrecon}, where it is shown that the statement of this Lemma is independent of the definition of $Q_{i+\frac12,j+\frac12}$. 

In 1-d, the point value of Active Flux could be directly linked to the numerical flux of DG (see Identification \ref{id:1d}). In multi-d, this has an analogy in the edge averages. The nodal point values, however, involve the $x$-flux and the $y$-flux simultaneously. Consider from now on therefore the following linear numerical fluxes:
\begin{align}
 \hat q_{i+\frac12,j}(y) &:= \alpha^+ q_{ij}\left(\frac{\Delta x}{2}, y\right ) + \alpha^- q_{i+1,j}\left(-\frac{\Delta x}{2}, y\right )\\
 \hat q_{i,j+\frac12}(x) &:=  \beta^+ q_{ij}\left(x, \frac{\Delta y}{2}\right ) + \beta^- q_{i,j+1}\left(x, -\frac{\Delta y}{2}\right ) 
\end{align}
Consistency requires $\alpha^+  + \alpha^- = \beta^+ + \beta^- = 1$. 
This allows to proceed to the following

\begin{ident}
Consider $K=1$. Starting from the DG approximation, define the degrees of freedom of tensorial AF as
\begin{align}
 Q_{i+\frac12,j+\frac12} &:= \alpha^+ \beta^+ q_{ij}\left(\frac{\Delta x}{2}, \frac{\Delta y}{2}\right) + \alpha^- \beta^+ q_{i+1,j}\left(-\frac{\Delta x}{2}, \frac{\Delta y}{2}\right)\label{eq:identification2dnodal} \\\nonumber &+ \alpha^+ \beta^- q_{i,j+1}\left(\frac{\Delta x}{2}, -\frac{\Delta y}{2}\right) + \alpha^- \beta^- q_{i+1,j+1}\left(-\frac{\Delta x}{2}, -\frac{\Delta y}{2}\right)
\end{align}
\begin{align}
 %\bar Q_{i+\frac12,j} &:= \frac{1}{\Delta y} \int_{-\frac{\Delta y}{2}}^{\frac{\Delta y}{2}} q_{ij}\left( \frac{\Delta x}{2}, y\right ) \, \dd y &
 \bar Q_{i+\frac12,j} &:= \frac{1}{\Delta y} \int_{-\frac{\Delta y}{2}}^{\frac{\Delta y}{2}} \hat q_{i+\frac12,j}(y) \, \dd y &
 %\bar Q_{i,j+\frac12} &:= \frac{1}{\Delta x} \int_{-\frac{\Delta x}{2}}^{\frac{\Delta x}{2}} q_{ij}\left( x, \frac{\Delta y}{2}\right ) \, \dd x
 \bar Q_{i,j+\frac12} &:= \frac{1}{\Delta x} \int_{-\frac{\Delta x}{2}}^{\frac{\Delta x}{2}} \hat q_{i,j+\frac12}(x) \, \dd x\label{eq:identification2dedge}
 \end{align}
 \begin{align}
 Q_{ij}^{(0)} := \frac{1}{\Delta x} \int_{-\frac{\Delta x}{2}}^{\frac{\Delta x}{2}}\frac{1}{\Delta y} \int_{-\frac{\Delta y}{2}}^{\frac{\Delta y}{2}} q_{ij}(x, y) \, \dd y \dd x \label{eq:identification2davg}
\end{align}
\end{ident}
This identification is reminiscent of the definition of the generalized Gauss-Radau projections in \cite{xu22}, which, however, do not themselves appear in this discussion.
The choice of the point value ensures that
\begin{align}
 \alpha^+ \hat q_{i,j+\frac12}\left(\frac{\Delta x}{2}\right) + \alpha^- \hat q_{i+1,j+\frac12}\left(-\frac{\Delta x}{2}\right) = \beta^+ \hat q_{i+\frac12,j}\left(\frac{\Delta y}{2}\right )  +    \beta^- \hat q_{i+\frac12,j+1}\left(-\frac{\Delta y}{2}\right ) = Q_{i+\frac12,j+\frac12} \label{eq:pointvalid}
\end{align}

\begin{remark} 
Consider $U^x, U^y > 0$ and the upwind flux. Then, 
\begin{align}
 Q_{i+\frac12,j+\frac12} := q_{ij}\left(\frac{\Delta x}{2}, \frac{\Delta y}{2}\right) 
\end{align}
\begin{align}
 \bar Q_{i+\frac12,j} &:= \frac{1}{\Delta y} \int_{-\frac{\Delta y}{2}}^{\frac{\Delta y}{2}} q_{ij}\left( \frac{\Delta x}{2}, y\right ) \, \dd y &
 \bar Q_{i,j+\frac12} &:= \frac{1}{\Delta x} \int_{-\frac{\Delta x}{2}}^{\frac{\Delta x}{2}} q_{ij}\left( x, \frac{\Delta y}{2}\right ) \, \dd x
 \end{align}
\end{remark}

The following Lemma establishes the update equations for linear combinations of some special values of the DG approximation:

\begin{lemma} \label{lem:flux}
 For any $\alpha,\beta,\gamma,\delta \in \mathbb R$ one has
 \begin{align}
  \frac{\dd}{\dd t} \frac{1}{\Delta y} &\int_{-\frac{\Delta y}{2}}^{\frac{\Delta y}{2}} \left(\alpha  q_{ij}\left(\frac{\Delta x}{2},y\right) + \beta  q_{i+1,j}\left(-\frac{\Delta x}{2},y\right)\right ) \,  \dd y \label{eq:intermedevoledgemomentx}
  \\\nonumber&+ \frac{1}{\Delta y} \int_{-\frac{\Delta y}{2}}^{\frac{\Delta y}{2}}  
  U^x \left(\alpha \del_x (q_{ij} + r^x_{ij})\Big |_{x=\frac{\Delta x}{2}} + \beta \del_x (q_{i+1,j} + r^x_{i+1,j})\Big |_{x=-\frac{\Delta x}{2}} \right ) \,  \dd y 
  \\\nonumber&+ \frac{1}{\Delta y}  U^y \left( \left[\alpha  \hat q_{i,j+\frac12}\left(\frac{\Delta x}{2}\right)     + \beta  \hat q_{i+1,j+\frac12}\left(-\frac{\Delta x}{2}\right)\right] - \left[\alpha \hat q_{i,j-\frac12}\left(\frac{\Delta x}{2}\right)   + \beta \hat q_{i+1,j-\frac12}\left(-\frac{\Delta x}{2}\right)  \right] \right )
  = 0
 \end{align}
 as well as 
 \begin{align}
 \frac{\dd}{\dd t} \frac{1}{\Delta x} &\int_{-\frac{\Delta x}{2}}^{\frac{\Delta x}{2}}  \left(\alpha q_{ij}\left(x, \frac{\Delta y}{2}  \right ) +\beta q_{i,j+1}\left( x, -\frac{\Delta y}{2}  \right ) \right )  \,  \dd x   \label{eq:intermedevoledgemomenty}
 \\\nonumber&+ \frac{1}{\Delta x}   U^x \left( \left[\alpha   \hat q_{i+\frac12,j}\left(\frac{\Delta y}{2}\right )  + \beta  \hat q_{i+\frac12,j+1}\left(-\frac{\Delta y}{2}\right )\right] - \left[ \alpha \hat q_{i+\frac12,j}\left(\frac{\Delta y}{2}\right )+ \beta \hat q_{i-\frac12,j+1}\left(-\frac{\Delta y}{2}\right ) \right]\right ) 
 \\\nonumber&+ \frac{1}{\Delta x} \int_{-\frac{\Delta x}{2}}^{\frac{\Delta x}{2}}  U^y \left( \alpha \del_y (q_{ij} + r^y_{ij}) \Big|_{y = \frac{\Delta y}{2}} + \beta \del_y (q_{i,j+1} + r^y_{i,j+1}) \Big|_{y = -\frac{\Delta y}{2}}  \right ) \,   \dd x  = 0 
\end{align}
and
\begin{align}
 \frac{\dd}{\dd t} & \left( \alpha q_{ij}\left(\frac{\Delta x}{2}, \frac{\Delta y}{2} \right) + \beta q_{i+1,j}\left(-\frac{\Delta x}{2}, \frac{\Delta y}{2} \right) + \gamma q_{i,j+1}\left(\frac{\Delta x}{2}, -\frac{\Delta y}{2} \right) + \delta q_{i+1,j+1}\left( -\frac{\Delta x}{2}, -\frac{\Delta y}{2}\right) \right )  \label{eq:intermedevolpoint}
 \\\nonumber & + U^x \left( \alpha \del_x (q_{ij} + r^x_{ij}) \Big |_{x=\frac{\Delta x}{2},y=\frac{\Delta y}{2}} +\beta \del_x (q_{i+1,j} + r^x_{i+1,j}) \Big |_{x=-\frac{\Delta x}{2},y=\frac{\Delta y}{2}}  \right . \\ \nonumber & \qquad \qquad \left. +\gamma \del_x (q_{i,j+1} + r^x_{i,j+1}) \Big |_{x=\frac{\Delta x}{2},y=-\frac{\Delta y}{2}} +\delta \del_x (q_{i+1,j+1} + r^x_{i+1,j+1}) \Big |_{x=-\frac{\Delta x}{2},y=-\frac{\Delta y}{2}} \right )
 \\\nonumber & + U^y \left(  \alpha  \del_y (q_{ij} + r^y_{ij}) \Big |_{x=\frac{\Delta x}{2},y=\frac{\Delta y}{2}} + \beta  \del_y (q_{i+1,j} + r^y_{i+1,j}) \Big |_{x=-\frac{\Delta x}{2},y=\frac{\Delta y}{2}}  \right . \\ \nonumber & \qquad \qquad \left.+ \gamma  \del_y (q_{i,j+1} + r^y_{i,j+1}) \Big |_{x=\frac{\Delta x}{2},y=-\frac{\Delta y}{2}}  + \delta  \del_y (q_{i+1,j+1} + r^y_{i+1,j+1}) \Big |_{x=-\frac{\Delta x}{2},y=-\frac{\Delta y}{2}}  \right ) = 0 
\end{align}
\end{lemma}
The proof can be found in Section \ref{sec:lemflux}.

Even though neither $q_{ij} + r_{ij}^x$, nor $q_{ij} + r_{ij}^y$ agree with the AF approximation $Q_{ij}$, in certain combinations they do agree along certain edges, as the following Lemma shows:

\begin{lemma} \label{lem:edgeequality}
 One has, for any $x \in [-\frac{\Delta x}{2}, \frac{\Delta x}{2}]$, and for $K=1$,
 \begin{align}
  \beta^+ (q_{ij} + r^x_{ij}) \Big |_{y=\frac{\Delta y}{2}} +\beta^- (q_{i,j+1} + r^x_{i,j+1}) \Big |_{y=-\frac{\Delta y}{2}} &= \beta^+ Q_{ij} \Big |_{y=\frac{\Delta y}{2}} + \beta^- Q_{i,j+1}\Big |_{y=\frac{\Delta y}{2}} \label{eq:identityedgex}
 \end{align}
 and analogously, for any $y \in [-\frac{\Delta y}{2}, \frac{\Delta y}{2}]$
 \begin{align}
  \alpha^+  (q_{ij} + r^y_{ij}) \Big |_{x=\frac{\Delta x}{2}} + \alpha^-   (q_{i+1,j} + r^y_{i+1,j}) \Big |_{x=-\frac{\Delta x}{2}} &= \alpha^+ Q_{ij}\Big |_{x=\frac{\Delta x}{2}}  + \alpha^- Q_{i+1,j} \Big |_{x=-\frac{\Delta x}{2}} 
  \end{align}
\end{lemma}
The proof can be found in Section \ref{sec:lemedgeequality}.

\begin{theorem}
 For $K = 1$, the DG method \eqref{eq:dg2d} is the Active Flux method \eqref{eq:updateedgeaverage}, \eqref{eq:updatenodal}--\eqref{eq:updateaverage} upon identifications \eqref{eq:identification2dnodal}--\eqref{eq:identification2davg}.
\end{theorem}

\begin{proof}
The evolution equation for edge averages is obtained from \eqref{eq:intermedevoledgemomentx} (Lemma \ref{lem:flux}) with $\alpha = \alpha^+$, $\beta = \alpha^-$:
\begin{align}
  \frac{\dd}{\dd t} \bar Q_{i+\frac12,j} &+ \frac{1}{\Delta y} \int_{-\frac{\Delta y}{2}}^{\frac{\Delta y}{2}}  
  U^x \left(\alpha^+ \del_x (q_{ij} + r^x_{ij})\Big |_{x=\frac{\Delta x}{2}} + \alpha^- \del_x (q_{i+1,j} + r^x_{i+1,j})\Big |_{x=-\frac{\Delta x}{2}} \right ) \,  \dd y 
  \\\nonumber&+ \frac{1}{\Delta y}  U^y \left( Q_{i+\frac12,j+\frac12} - Q_{i+\frac12,j-\frac12} \right )
  = 0
 \end{align}
 where 
 \eqref{eq:pointvalid}
 has been used. Moreover, due to the integral property of Radau polynomials and the tensor-product structure,
 \begin{align}
  \int_{-\frac{\Delta y}{2}}^{\frac{\Delta y}{2}} \del_x (q_{ij} + r^x_{ij}) \,\dd y = \int_{-\frac{\Delta y}{2}}^{\frac{\Delta y}{2}} \del_x Q_{ij}  \,\dd y 
 \end{align}
 Then, this equation indeed becomes the natural update \eqref{eq:updateedgeaverage}, with a regularization of the $x$-derivative by means of the same upwinding as employed in the numerical flux of DG. Analogous arguments apply for the perpendicular edge.

 For the point value, consider \eqref{eq:intermedevolpoint} with $\alpha = \alpha^+ \beta^+, \beta = \alpha^- \beta^+ ,\gamma = \alpha^+ \beta^- ,\delta = \alpha^- \beta^-$:
 \begin{align}
 \frac{\dd}{\dd t} &Q_{i+\frac12,j+\frac12}
 \\\nonumber & + U^x \left[ \alpha^+ \del_x \left(  \beta^+(q_{ij} + r^x_{ij}) \Big |_{y=\frac{\Delta y}{2}} + \beta^- (q_{i,j+1} + r^x_{i,j+1}) \Big |_{y=-\frac{\Delta y}{2}} \right )\Big |_{x=\frac{\Delta x}{2}} \right . \\ \nonumber & \qquad \qquad \left.
 + \alpha^-\del_x \left(  \beta^+ (q_{i+1,j} + r^x_{i+1,j}) \Big |_{y=\frac{\Delta y}{2}}   + \beta^- (q_{i+1,j+1} + r^x_{i+1,j+1}) \Big |_{y=-\frac{\Delta y}{2}} \right)\Big |_{x=-\frac{\Delta x}{2}}  \right ]
 \\\nonumber & + U^y \left[  \beta^+ \del_y\left( \alpha^+   (q_{ij} + r^y_{ij}) \Big |_{x=\frac{\Delta x}{2}} + \alpha^-    (q_{i+1,j} + r^y_{i+1,j}) \Big |_{x=-\frac{\Delta x}{2}} \right ) \Big |_{y=\frac{\Delta y}{2}}  \right . \\ \nonumber & \qquad \qquad \left.+  \beta^- \del_y\left( \alpha^+ (q_{i,j+1} + r^y_{i,j+1}) \Big |_{x=\frac{\Delta x}{2}}  + \alpha^-  (q_{i+1,j+1} + r^y_{i+1,j+1}) \Big |_{x=-\frac{\Delta x}{2}} \right )\Big |_{y=-\frac{\Delta y}{2}}  \right ] = 0 
\end{align}
 
With Lemma \ref{lem:edgeequality} this equation can be rewritten as
\begin{align}
 \frac{\dd}{\dd t} &Q_{i+\frac12,j+\frac12}
 \\\nonumber & + U^x \left[ \alpha^+ \del_x \left(  \beta^+Q_{ij} \Big |_{y=\frac{\Delta y}{2}} + \beta^- Q_{i,j+1} \Big |_{y=-\frac{\Delta y}{2}} \right )\Big |_{x=\frac{\Delta x}{2}} \right . \\ \nonumber & \qquad \qquad \left.
 + \alpha^-\del_x \left(  \beta^+ Q_{i+1,j} \Big |_{y=\frac{\Delta y}{2}}   + \beta^- Q_{i+1,j+1} \Big |_{y=-\frac{\Delta y}{2}} \right)\Big |_{x=-\frac{\Delta x}{2}}  \right ]
 \\\nonumber & + U^y \left[  \beta^+ \del_y\left( \alpha^+   Q_{ij} \Big |_{x=\frac{\Delta x}{2}} + \alpha^-    Q_{i+1,j} \Big |_{x=-\frac{\Delta x}{2}} \right ) \Big |_{y=\frac{\Delta y}{2}}  \right . \\ \nonumber & \qquad \qquad \left.+  \beta^- \del_y\left( \alpha^+ Q_{i,j+1} \Big |_{x=\frac{\Delta x}{2}}  + \alpha^-  Q_{i+1,j+1} \Big |_{x=-\frac{\Delta x}{2}} \right )\Big |_{y=-\frac{\Delta y}{2}}  \right ] = 0 
\end{align}
 which indeed is the natural upwind point value update involving a weighted sum of the one-sided derivatives to do multi-dimensional upwinding, and generalizing \eqref{eq:updatenodal}.

The update equation for the cell average follows from \eqref{eq:vwupdateintermiedate} with the choice $v = w = 1$:
\begin{align}
 0 &= \frac{\dd}{\dd t} Q_{ij}^{(0)}  + \frac{1}{\Delta x} \int_{-\frac{\Delta x}{2}}^{\frac{\Delta x}{2}} \frac{1}{\Delta y} \int_{-\frac{\Delta y}{2}}^{\frac{\Delta y}{2}}  \left(U^x \del_x (q_{ij} + r^x_{ij}) + U^y \del_y (q_{ij} + r^y_{ij})\right) \,  \dd y \dd x  \\
 &= \frac{\dd}{\dd t} Q_{ij}^{(0)}  +  U^x \frac{1}{\Delta x} \frac{1}{\Delta y} \int_{-\frac{\Delta y}{2}}^{\frac{\Delta y}{2}} \left(q_{ij}+ r^x_{ij}\right )\Big |_{x = -\frac{\Delta x}{2}}^{\frac{\Delta x}{2}} \,  \dd y  + U^y \frac{1}{\Delta x} \frac{1}{\Delta y}\int_{-\frac{\Delta x}{2}}^{\frac{\Delta x}{2}}  \left(q_{ij} + r^y_{ij} \right) \Big|_{-\frac{\Delta y}{2}}^{\frac{\Delta y}{2}} \,  \dd y 
\end{align}

Here, again the integral properties of the Radau polynomials allow to rewrite this as
\begin{align}
 0 &= \frac{\dd}{\dd t} Q_{ij}^{(0)}  +  U^x \frac{1}{\Delta x} \frac{1}{\Delta y} \int_{-\frac{\Delta y}{2}}^{\frac{\Delta y}{2}} Q_{ij}\Big |_{x = -\frac{\Delta x}{2}}^{\frac{\Delta x}{2}} \,  \dd y  + U^y \frac{1}{\Delta x} \frac{1}{\Delta y}\int_{-\frac{\Delta x}{2}}^{\frac{\Delta x}{2}}  Q_{ij} \Big|_{-\frac{\Delta y}{2}}^{\frac{\Delta y}{2}} \,  \dd y 
\end{align}
which is the usual update \eqref{eq:updateaverage} of the cell average for AF.

\end{proof}

\begin{remark}
 The zeros of $R_{\text{L}, \Delta x}^{K+1}$ are straight lines parallel to the $y$-axis, and those of $R_{\text{L}, \Delta y}^{K+1}$ are parallel the $x$-axis. At the crossing points, i.e. at all points
 \begin{align}
  \{ (x, y) : R_{\text{L}, \Delta x}^{K+1}(x) = 0 \text{ and } R_{\text{L}, \Delta y}^{K+1}(y) = 0 \}
 \end{align}
 one has $q_{ij} = Q_{ij}$, such that the $K+2$ order of accuracy of AF becomes immediately visible as superconvergence of DG.
\end{remark}

\section{Performance Comparison} \label{sec:perfromance}

In this section we investigate the performance of semi-discrete Active Flux methods in two space dimensions of different orders, and compare to that of tensorial modal Discontinuous Galerkin methods. We focus on Active Flux with the minimal possible number of degrees of freedom (point values at the cell interfaces and moments), as proposed in \cite{lechner25} (serendipity-type, called therein $\mathrm{GenAF}(M_{\symtri})$). This is not the new tensorial version for which equivalence for DG has been shown above, but we prefer to compare with this more established version. Both methods are implemented in a unified way, using identical routines for equivalent computational components wherever applicable. For example, both types of methods use moments as degrees of freedom, and thus, the computation of them is done in a similar way. Also, we ensured that, for instance, the computational effort to compute the reconstruction is of the same order of magnitude for both types of methods.

In the following, $N_\mathrm{order}$ denotes the order of the Active Flux or Discontinuous Galerkin methods, $N_{\mathrm{dofs}}$ the exclusive degrees of freedom per cell and $N_{\mathrm{tdofs}}$ the number of degrees of freedom accessible per cell. This distinction is needed, since the point values of the Active Flux methods are shared between neighbouring cells. The number of moments, edge point values and node point values accessible per cell are named $N_{\mathrm{mom}}$, $N_{\mathrm{edge}}$ and $N_{\mathrm{node}}$. 
 
 For the modal Discontinuous Galerkin methods we have the following relation:
 \begin{equation*}
  	N_{\mathrm{tdofs}} = N_{\mathrm{dofs}} = N_{\mathrm{mom}} = (N_{\mathrm{order}})^2
 \end{equation*}
 The number of moments and point values per cell of the Active Flux methods is given by
 \begin{align*}
 	N_{\mathrm{node}} = 4, &&
 	N_{\mathrm{edge}} = 4(N_{\mathrm{order}} -2) &&
 	\text{and}&&
 	N_{\mathrm{mom}}  = \max\left(1 , \frac{1}{2}(N_{\mathrm{order}} - 4)(N_{\mathrm{order}} - 3)\right) 
 \end{align*}
 and the number of (exclusive and accessible) degrees of freedom per cell is given by
 \begin{align*}
 	N_{\mathrm{dofs}} &= \frac{1}{4}N_{\mathrm{node}} + \frac{1}{2}N_{\mathrm{edge}} + N_{\mathrm{mom}} && \text{and} && N_{\mathrm{tdofs}} = N_{\mathrm{node}} + N_{\mathrm{edge}} + N_{\mathrm{mom}}.
 \end{align*}
 To compute the moments, integrals need to be evaluated numerically, which is done by a Gauss-Legendre quadrature, whose order is chosen as the sum of the order of the method ($N_\mathrm{order}$) and the largest exponent of the computed moment. If this value is even, then the next smallest odd number is chosen. As a time discretisation, we use a three-stage, third-order and a five-stage, fourth-order strong stability-preserving Runge-Kutta scheme. In Table \ref{tab:overview} an overview of the methods considered here is given, where the highest order of integration used by a method of a given order is denoted as $N_{\mathrm{int}}$ and the CFL number as $C_{\mathrm{CFL}}$.  In our case, the size of the time step is given by $\Delta t = C_{\mathrm{CFL}}\Delta x = N_\mathrm{order}^{-\frac{1}{2}}$.

\begin{table}[h!]
	\centering
	\begin{tabular}{| m{1cm} | m{0.75cm} | m{0.75cm} |m{0.75cm} |m{0.75cm} |m{0.75cm} |m{0.75cm} |m{0.75cm} |m{0.75cm} |  m{0.75cm} | m{0.75cm} |}
		\hline
		& \multicolumn{5}{c|}{\text{Active Flux Method}}  & \multicolumn{5}{c|}{\text{Discontinuous Galerkin Method}} \\
		\hline
		$N_{\mathrm{order}}$ & 3 & {4} & {5} & 6 & 7 & 2 & 3 & {4} & {5} & 6 \\

		\hline
		\textcolor{black}{$N_{\mathrm{dofs}}$}& 4 & {6} & {8} & 12 & 17  & 4 & 9 & {16} & 25 & 36 \\
		\hline
		$N_{\mathrm{tdofs}}$& 9 & {13} & {17} & 23 & 30  & 4 & 9 & {16} & 25 & 36\\
		\hline 
		$N_{\mathrm{mom}}$ & 1 & {1} & {1} & 3 & 6   & 4 & 9 & {16} & 25 & 36\\
		\hline 
		$N_{\mathrm{edge}}$ & 4 & {8} & {12} & 16 & 20 & - & - & - & - & -\\
		\hline 
		$N_{\mathrm{node}}$ & 4 & {4} & {4} & 4 & 4  & - & - & - & - & -\\
		\hline
		$N_{\mathrm{int}}$ & 3 & {3} & {5} & 7 & 9 & 3 & 5 & {7} & 9 & 11\\
		\hline
		$C_{\mathrm{CFL}}$ & 0.27 & 0.2  & 0.17  & 0.12 & 0.085 & 0.2 & 0.1 & 0.05 & 0.02 & 0.01\\
		\hline
	\end{tabular}
	\caption{Overview of the different Active Flux and Discontinuous Galerkin methods.}
	\label{tab:overview}
\end{table}

\subsection{Theoretical Expectations}
Before we compare the measured results, we will investigate the theoretically expected asymptotic behaviour of the Active Flux and Discontinuous Galerkin methods. First, we consider the runtime, which will be denoted by $\tau$, as a function of the grid and of the method. Using the unified implementation, we expect 
\begin{align}
	\tau = \mathcal{O}( N_{\mathrm{timesteps}} \cdot N_{\mathrm{stages}} \cdot N_{\mathrm{cells}} \cdot \max(N_\mathrm{mom} \cdot N_\mathrm{int} \cdot N_\mathrm{tdofs}, N_{\mathrm{mom}} \cdot (N_{\mathrm{int}})^2 \cdot N_{\mathrm{tdofs}} )),
	\label{eq:runtime_dg}
\end{align}
to be the asymptotic behaviour of the runtime of the Discontinuous Galerkin methods, where $N_{\mathrm{cells}}$ is the number of grid cells, $N_{\mathrm{timesteps}}$ the number of time steps and $N_{\mathrm{stages}}$ the number of stages of the time discretisation. $N_{\mathrm{timesteps}}$, $N_{\mathrm{stages}}$, $N_{\mathrm{tdofs}}$ and $N_{\mathrm{int}}$ depend on the specific method and $N_{\mathrm{timesteps}}$ and $N_{\mathrm{cells}}$ on the grid.
For the Active Flux methods, we expect a slightly different asymptotic behaviour: 
\begin{align}
	\label{eq:runtime_af}
	\tau = \mathcal{O}( N_{\mathrm{timesteps}} \cdot N_{\mathrm{stages}} \cdot N_\mathrm{cells} \cdot  \max &( N_\mathrm{mom} \cdot N_\mathrm{int} \cdot N_\mathrm{tdofs}, \\ \nonumber & \  N_{\mathrm{mom}} \cdot (N_{\mathrm{int}})^2 \cdot N_{\mathrm{tdofs}}, N_{\mathrm{tdofs}}, N_{\mathrm{edge}} \cdot N_{\mathrm{tdofs}})),
\end{align}
where the number of the different degrees of freedom depends on the method.
If we fix the used method (varying $N_\mathrm{cells}$ and thus $N_\mathrm{timesteps}$ in \eqref{eq:runtime_dg} and \eqref{eq:runtime_af} only) and therefore consider $\tau$ only as a function of the grid, we expect for the Active Flux and the Discontinuous Galerkin methods:
\begin{align*}
	\tau =  \mathcal{O}\left((N_{\mathrm{cells}})^{1.5}\right),
\end{align*}
since the number of time steps depends on the chosen spatial step size and $C_{\mathrm{CFL}}$, which is fixed for each method. Another option is to consider the runtime as a function of the numerical error which again depends on the spatial step size. Let $\mathcal{E}_{\mathrm{dofs}}$ be the maximum over all $l^2$-errors of the degrees of freedom of a given method. If we consider $\tau$ to be a function of $\mathcal{E}_{\mathrm{dofs}}$ and assume that $\mathcal{E}_{\mathrm{dofs}}$ is unaffected by the error introduced by the time discretisation and thereby dominated only by the spatial error, we obtain the following expected asymptotic behaviour of the runtime: 
\begin{align*}
	\tau = \mathcal{O}\left( \left(\mathcal{E}_{\mathrm{dofs}}\right)^{-\frac{3}{N_\mathrm{order}}}\right).
\end{align*}

Next, we will consider methods of different orders on a fixed grid and with a fixed time discretisation. We define $\tau_s$ to be the runtime per time step and obtain for the Discontinuous Galerkin methods from \eqref{eq:runtime_dg} and \eqref{eq:runtime_af}
 \begin{align*}
 	\tau_s = \mathcal{O}((N_\mathrm{tdofs})^3) =\mathcal{O}((N_\mathrm{dofs})^3)= \mathcal{O}((N_\mathrm{order})^6),
 \end{align*}
 if we consider the runtime per time step as a function of the number of degrees of freedom or the order of a specific method. For the Active Flux methods we expect a similar asymptotic behaviour: 
\begin{align*}
	\tau_s = \mathcal{O}((N_\mathrm{order})^6).
\end{align*}
If we consider the runtime per time step as a function of the number of degrees of freedom, we get
\begin{align*}
	\tau_s = \mathcal{O}((N_\mathrm{tdofs})^3) && \text{and} && \tau_s = \mathcal{O}((N_\mathrm{dofs})^3).
\end{align*}
Concerning the total runtime $\tau$, we expect for the Discontinuous Galerkin and the Active Flux methods
\begin{align*}
	\tau = \tau_s \cdot \mathcal{O}(N_{\mathrm{timesteps}}),
\end{align*}
where the number of time steps depends solely on $C_{\mathrm{CFL}}$. Hence, the asymptotic relation between the total runtime and the order or the number of degrees of freedom per cell depends on the asymptotic relation between $C_{\mathrm{CFL}}$ and $N_\mathrm{order}$, $N_\mathrm{dofs}$ or $N_\mathrm{tdofs}$ respectively.

\subsection{Numerical Results}

For our performance study, we consider the linear advection equation
\begin{align*}
	\partial_t q + \partial_x q+\partial_yq = 0
\end{align*}
on $[0,1]^2$ with Dirichlet boundary conditions and smooth initial data
\begin{align*}
	q_0(x,y) = 0.8 + \exp \left(-\left(\frac{x-0.5}{0.05}\right)^2-\left(\frac{y-0.5}{0.05}\right)^2\right).
\end{align*}
The spatial discretisation uses Cartesian grids with an equal number of cells in  $x$ and $y$ directions (i.e., $\Delta x = \Delta y$). To obtain the numerical results, we compute the solution to this problem at time $T = 0.1$ with the different Active Flux and Discontinuous Galerkin methods and measure the total runtime, average runtime per time step and the $l^2$-error of the different degrees of freedom. In the following we will refer, for example, to a Discontinuous Galerkin or an Active Flux method of order five with a SSPRK3 time discretisation as DG53 or AF53.
For the first numerical test, we consider the runtime $\tau$ as a function of $N_\mathrm{cells}$. In Figure \ref{fig:runtime_cellnumber} one can see the measured runtime for the different methods with different values of $N_\mathrm{cells}$ (i.e., $20^2$, $40^2$, $80^2$ and $160^2$). As one can see, the Active Flux methods are overall significantly faster than the Discontinuous Galerkin methods with the same order. Note that the curves for AF43 and DG23 and AF44 and DG24 lie on top of each other. In addition, the measured runtime behaves as expected.

\begin{figure}[h!]
	\centering
	\includegraphics[scale = 0.6]{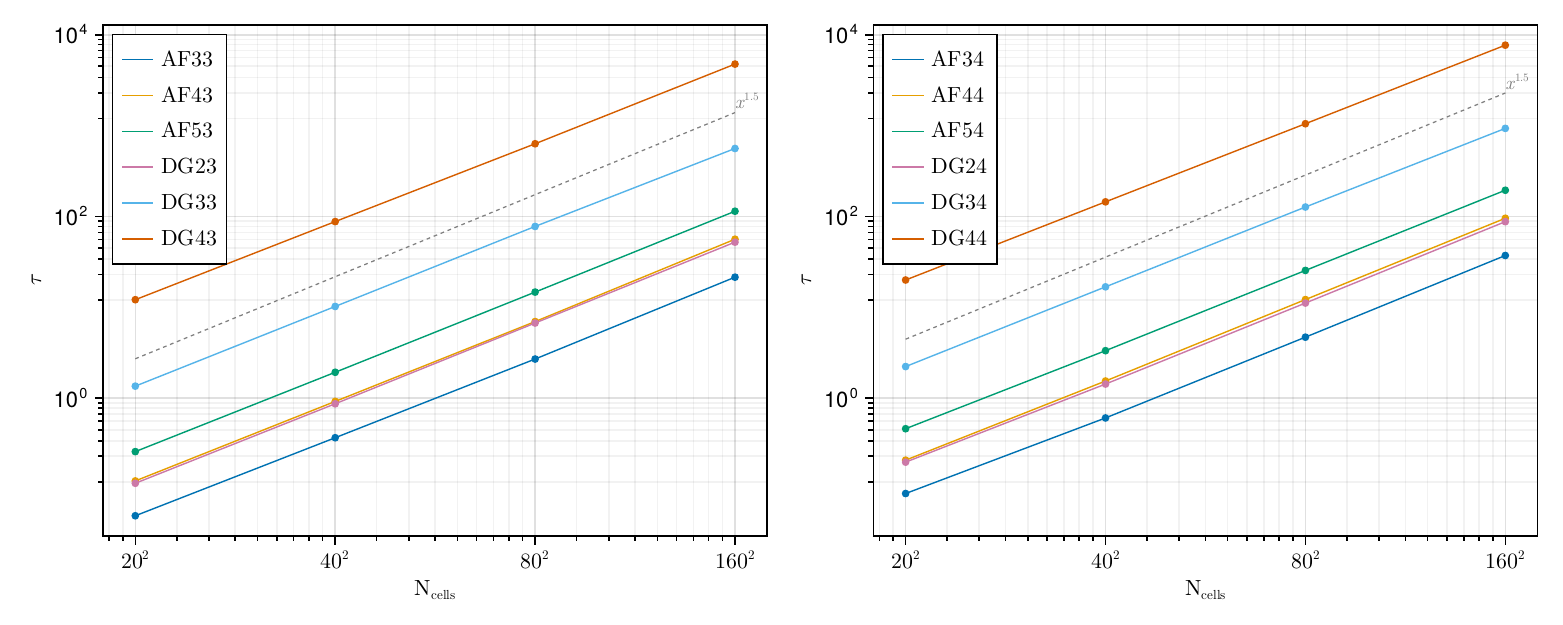}
	\caption{The total runtime of the Active Flux and Discontinuous Galerkin methods of differnt order with respect to the number of cells.}
	\label{fig:runtime_cellnumber}	
\end{figure}

In Figure \ref{fig:runtime_error} one can see the runtime with respect to $\mathcal{E}_\mathrm{dofs}$. For a given error, the Active Flux methods are capable of computing the numerical results within a shorter runtime. The most runtime-efficient methods with respect to the accuracy are AF33, AF43 and AF54. Furthermore, note that the runtime of AF43 and AF53 appears to have a different behaviour than expected. The reason for this is that the error of methods is presumably dominated by the error of the time discretisation. As one can see in Table \ref{tab:eoc1}, the experimental order of convergences of both methods differs from the formal order. By using a fourth-order time discretisation, the fourth and fifth-order Active Flux methods achieve an experimental order of convergence of four and five. This is not needed for the Discontinuous Galerkin method of order four, as the experimental order of convergence for DG43 is already approximately four (see Table \ref{tab:eoc2}). 

\begin{figure}[h!]
	\centering
	\includegraphics[scale = 0.6]{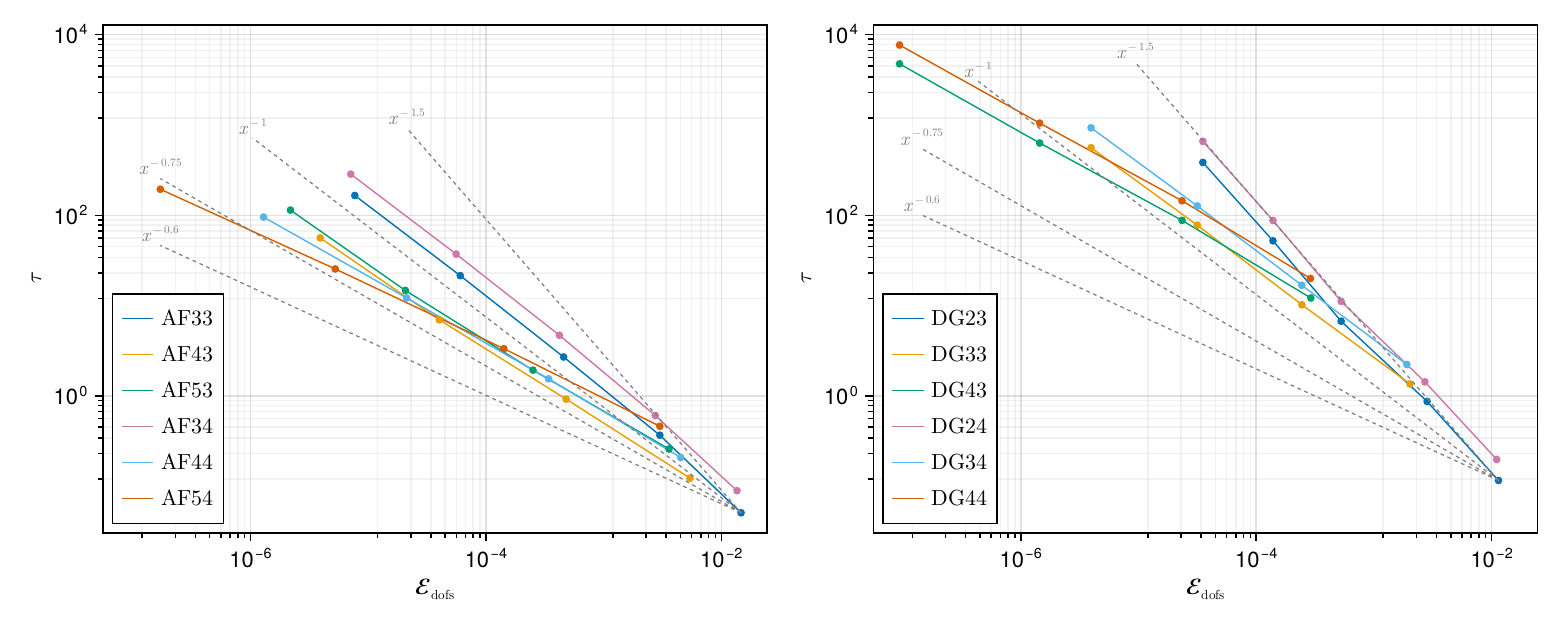}
	\caption{The total runtime of the Active Flux and Discontinuous Galerkin methods of differnt with respect to the error $\mathcal{E}_\mathrm{dofs}$.}
	\label{fig:runtime_error}	
\end{figure}

\begin{table}[h!]
	\centering
	\begin{tabular}{| m{1.5cm} | m{1cm} | m{1cm} |m{1cm} |m{1cm} | m{1cm} |  m{1cm} |}
		\hline
		$\Delta x$ & AF33 & AF43 & AF53 & AF34 & AF44 & AF54 \\
		\hline
		0.05 & - & - & - & - & - & - \\
		\hline
		0.025 & 2.2924 & 3.4956 & 3.8381 & 2.3038 & 3.7278 & 4.3997  \\
		\hline
		0.0125 & 2.7147 & 3.5786 & 3.5996 & 2.7077 & 4.0076 & 4.761  \\
		\hline
		0.0625 & 2.9172 & 3.3618 & 3.2449 & 2.9138 & 4.0369 & 4.9365 \\
		\hline
		0.003125 & 2.9762 & - &  - & 2.9743 & - & - \\
		\hline
	\end{tabular}
	\caption{Experimental order of convergence of the different Active Flux methods.}
	\label{tab:eoc1}
\end{table}

\begin{table}[h!]
	\centering
	\begin{tabular}{| m{1.5cm} | m{1cm} | m{1cm} |m{1cm} |m{1cm} | m{1cm} |  m{1cm} |}
		\hline
		$\Delta x$ & DG23 & DG33 & DG43 & DG24 & DG34 & DG44 \\
		\hline
		0.05 & - & - & - & - & - & -  \\
		\hline
		0.025 & 2.0138 & 3.0501 & 3.6359 & 2.0268 & 2.9636 & 3.6243 \\
		\hline
		0.0125 & 2.4265 & 2.9529 & 4.0162 & 2.364 & 2.9496 & 4.0179 \\
		\hline
		0.0625 & 1.9227 & 2.9984 & 3.9554 & 1.9225 & 2.9983 & 3.9536 \\
		\hline
		0.003125 & 1.98 & - &  - & 1.9801 & - & -  \\
		\hline
	\end{tabular}
	\caption{Experimental order of convergence of the different Discontinuous Galerkin methods.}
	\label{tab:eoc2}
\end{table}

Next, we consider the runtime as a function of the order or the number of exclusive degrees of freedom per cell and therefore all numerical results are computed on a $40 \times 40$ grid with $\Delta x = 0.0125$. In Figure \ref{fig:runtime_order} one can see the measured average runtime per time step $\tau_s$ and the total runtime $\tau$ with respect to the orders of the methods. For a given order, the average runtime per time step of the Active Flux methods considered here is shorter than the average runtime per time step of the Discontinuous Galerkin methods. The difference is even greater for the total runtime, because the Active Flux methods allow for a larger CFL number and therefore have fewer time steps. The steep slope between the Active Flux methods of order five and six is due to the introduction of new moments and the associated numerical computation of the integrals, which dominates the runtime. Up to order five, the Active methods only have the zeroth moment (i.e., cell average) as a degree of freedom, whose computation is significantly cheaper than the computation of higher-order moments. Moreover, the measured results for the average runtime per time step correspond approximately to the theoretical expectations. 

\begin{figure}[h!]
	\centering
	\includegraphics[scale = 0.6]{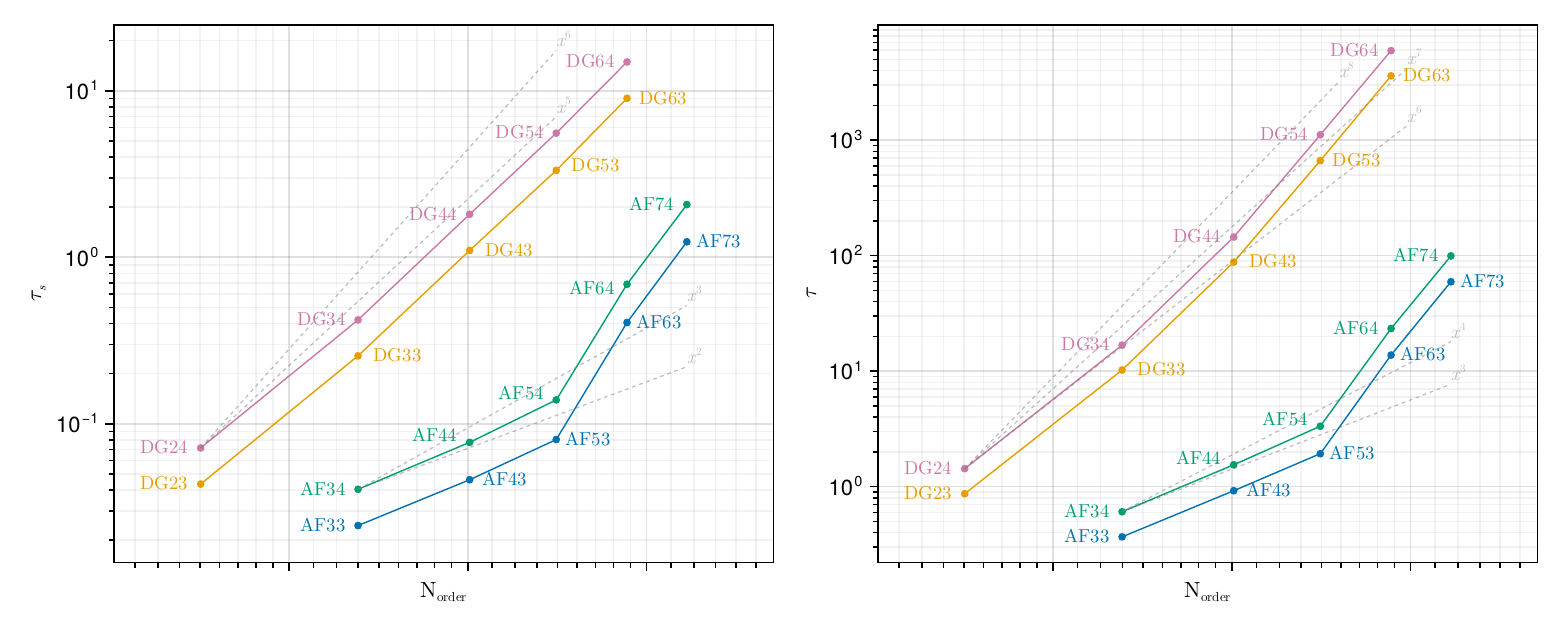}
	\caption{The average runtime per time step and the total runtime with respect to order of the Active Flux and Discontinuous Galerkin methods.}
	\label{fig:runtime_order}	
\end{figure}

In Figure \ref{fig:runtime_dofs}, one can see the runtime with respect to the number of exclusive degrees of freedom per cell, which corresponds to the memory usage. The average runtime per time step of the Active Flux methods up to order five is faster than the average runtime per time step of the Discontinuous Galerkin methods. For order six and higher, the average runtime per time step of the Active Flux methods is at least as fast as the average runtime per time step of the Discontinuous Galerkin methods with the same time discretisation. The total runtime of the different Active Flux methods is faster than the total runtime of the Discontinuous Galerkin methods. The reason for this is again the larger CFL number of the Active Flux methods. In addition, for a given number of degrees of freedom per cell, the Active Flux methods compute more accurate results, as one can see in Figure \ref{fig:error_dofs}. Especially the error for the Active Flux methods with a fourth-order time discretisation is significantly lower than the error for the Discontinuous Galerkin methods with a similar number of exclusive degrees of freedom per cell. In contrast, for a given order, the Discontinuous Galerkin methods are more accurate than the Active Flux methods.

\begin{figure}[h!]
	\centering
	\includegraphics[scale = 0.6]{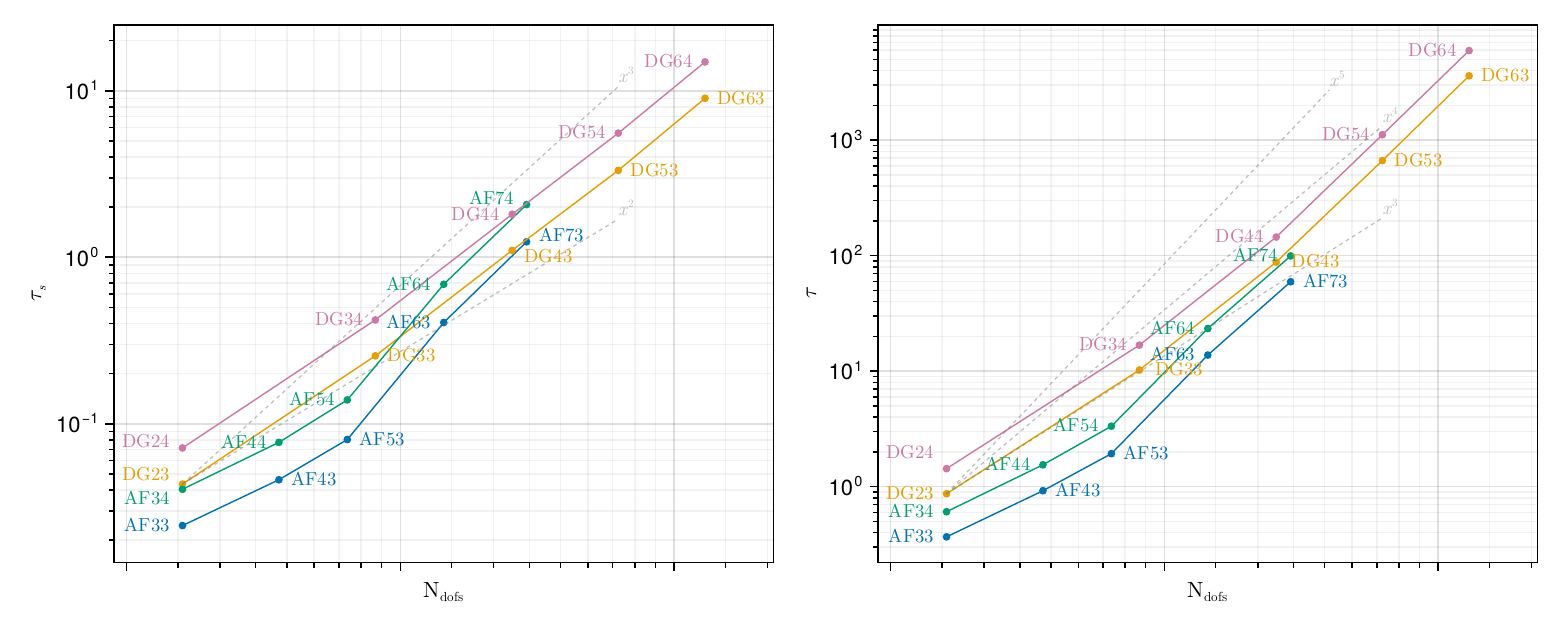}
	\caption{The average runtime per time step and the total runtime with respect to number of exclusive degrees of freedom of the Active Flux and Discontinuous Galerkin methods}
	\label{fig:runtime_dofs}	
\end{figure}

\begin{figure}[h!]
	\centering
	\includegraphics[scale = 0.6]{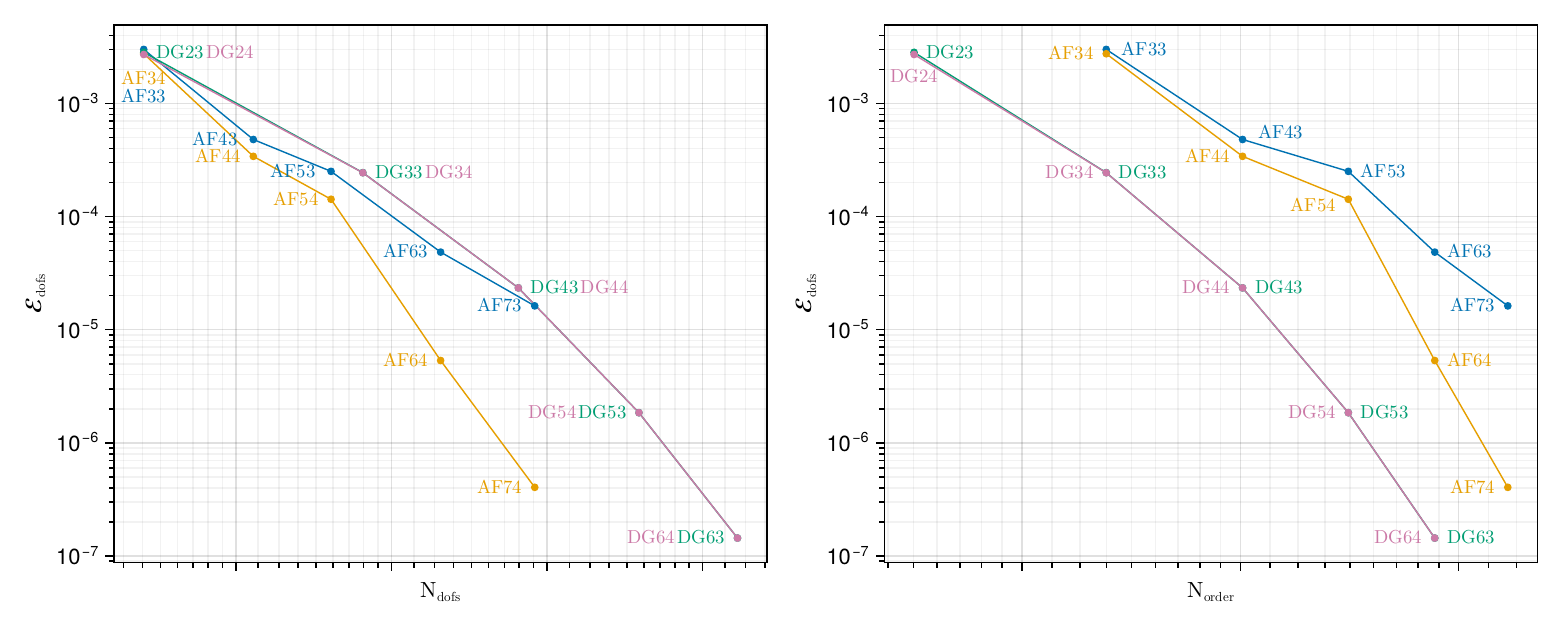}
	\caption{The $l^2$-error with respect to the the number of degrees of freedom and the order of the different Active Flux and Discontinuous Galerkin methods.}
	\label{fig:error_dofs}	
\end{figure}

One way to find the most efficient method with respect to memory usage, accuracy and runtime out of all the methods considered here is to find the one minimising 
\begin{equation*}
	N_\mathrm{dof} \cdot \mathcal{E}_\mathrm{dof} \cdot \tau.
\end{equation*}
In Table \ref{tab:eff} this value for the different methods at different grid sizes is displayed, and the smallest value is highlighted. If one measures efficiency in this way, the Active Flux methods are superior to the Discontinuous Galerkin methods. 

\begin{table}[h!]
	\centering
	\begin{tabular}{| m{1cm} | m{1.5cm} | m{1.5cm} |m{1.5cm} |m{1.5cm} |}
		\hline 
		 & $20 \times 20$ & $40 \times 40$ & $80 \times 80$ & $160 \times 160$ \\
		\hline
		DG23 & 0.0053 & 0.0098 & 0.0142 & 0.029 \\
		\hline
		DG24 & 0.0087 & 0.0155 & 0.0235 & 0.0487 \\
		\hline
		\textbf{AF33} & \textbf{0.003} & 0.0044 & 0.0049 & 0.0052 \\
		\hline
		AF34 & 0.0049 & 0.0067 & 0.0079 & 0.0083 \\
		\hline
		DG33 & 0.0248 & 0.0225 & 0.0221 & 0.0199 \\
		\hline
		DG34 & 0.0382 & 0.0369 & 0.0361 & 0.0332 \\
		\hline
		AF43 & 0.004 & 0.0027 & 0.0017 & 0.0013 \\
		\hline
		AF44 & 0.0056 & 0.0032 & 0.0015 & 0.0007\\
		\hline
		DG43 & 0.0566 & 0.0329 & 0.0146 & 0.0071 \\
		\hline
		DG44 & 0.0924 & 0.0543 & 0.0243 & 0.0115 \\
		\hline
		AF53 & 0.0074 & 0.0039 & 0.0024 & 0.002 \\
		\hline
		\textbf{AF54} & 0.0111 & 0.0038 & \textbf{0.0011} & \textbf{0.0003} \\
		\hline
		DG53 & 0.1924 & 0.0307 & - & - \\
		\hline
		DG54 & 0.3185 & 0.0514 & - & - \\
		\hline
		AF63 & 0.023 & 0.008 & - & - \\
		\hline
		AF64 & 0.0281 & 0.0015 & - & - \\
		\hline
		DG63 & 0.0905 & 0.0187 & - & - \\
		\hline
		DG64 & 0.0216 & 0.0164 & - & - \\
		\hline
		AF73 & 0.1284 & 0.0965 & - & - \\
		\hline
		\textbf{AF74} & 0.0077 & \textbf{0.0007} & - & - \\
		\hline
	\end{tabular}
	\caption{The value $N_\mathrm{dof} \cdot \mathcal{E}_\mathrm{dof} \cdot \tau.
$ for the different methods and different gird sizes.}
	\label{tab:eff}
\end{table}

An overview of the measured results for the different methods is given in Table \ref{tab:overview2}, where the number of degrees of freedom, total runtime and $l^2$-error for different grids are stated with respect to the results for AF54. 

\begin{table}[h!]
	\centering
	\begin{tabular}{| m{0.9cm} | m{0.9cm} | m{1.3cm} |m{1.3cm} |m{1.3cm} | m{1.3cm} |  m{1.3cm} | m{1.3cm} | m{1.3cm} | m{1.3cm} |}
		\hline
		 & & \multicolumn{2}{c|}{$20 \times 20$ } & \multicolumn{2}{c|}{$40 \times 40$} & \multicolumn{2}{c|}{$80 \times 80$} & \multicolumn{2}{c|}{$160 \times 160$} \\
		\hline
		 &   $N_\mathrm{dofs}$ &  $\tau$ &  $\mathcal{E}_\mathrm{dofs}$ &  $\tau$ &  $\mathcal{E}_\mathrm{dofs}$&  $\tau$ & $\mathcal{E}_\mathrm{dofs}$ &  $\tau$ & $\mathcal{E}_\mathrm{dofs}$ \\
		\hline
		DG23 & 0.5 & 0.251 & 3.8065 & 0.2604 & 19.8958 & 0.2638 & 100.3551 & 0.2683 & 810.5553 \\
		\hline
		DG24 & 0.5 & 0.4288 & 3.6782 & 0.4291 & 19.0528 & 0.4376 & 100.353 & 0.451 & 810.6546 \\
		\hline
		AF33 & 0.5 & 0.1101 & 4.9001 & 0.1101 & 21.1132 & 0.1059 & 87.2066 & 0.1105 & 353.5169 \\
		\hline
		AF34 & 0.5 & 0.1935 & 4.5234 & 0.1818 & 19.3376 & 0.1844 & 80.2651 & 0.1913 & 326.1563 \\
		\hline
		DG33 & 1.125 & 2.9412 & 0.6746 & 3.0652 & 1.7192 & 3.0495 & 6.0202 & 2.881 & 23.0699 \\
		\hline
		DG34 & 1.125 & 4.8306 & 0.6339 & 5.0378 & 1.7152 & 4.9861 & 6.0201 & 4.7938 & 23.0698 \\
		\hline
		AF43 & 0.75 & 0.2663 & 1.8042 & 0.2764 & 3.3763 & 0.2735 & 7.6629 & 0.289 & 22.8259 \\
		\hline
		AF44 & 0.75 & 0.4506 & 1.5042 & 0.4634 & 2.3965 & 0.4778 & 4.0401 & 0.4915 & 7.5368  \\
		\hline
		DG43 & 2.0 & 26.2849 & 0.097 & 26.3351 & 0.1647 & 24.7744 & 0.276 & 24.4461 & 0.5448  \\
		\hline
		DG44 & 2.0 & 43.2942 & 0.0962 & 43.4175 & 0.1647 & 41.1555 & 0.2757 & 39.4917 & 0.5449  \\
		\hline
		AF53 & 1.0 & 0.5596 & 1.1959 & 0.5783 & 1.7651 & 0.5776 & 3.9481 & 0.5874 & 12.7529 \\
		\hline
		AF54 & 1.0 & 1.0 & 1.0 & 1.0 & 1.0 & 1.0 & 1.0 & 1.0 & 1.0 \\
		\hline
		DG53 & 3.125 & 207.3767 & 0.0268 & 199.2672 & 0.013 &- &- &- &- \\
		\hline
		DG54 & 3.125 & 343.8791 & 0.0267 & 333.4801 & 0.013 &- &- &- &-  \\
		\hline
		AF63 & 1.5 & 3.9886 & 0.3464 & 4.1296 & 0.3412 &- &- &- &- \\
		\hline
		AF64 & 1.5 & 6.59 & 0.2562 & 7.0084 & 0.0376 &- &- &- &-  \\
		\hline
		DG63 & 4.5 & 1084.433 & 0.0017 & 1079.589 & 0.001 &- &- &- &- \\
		\hline
		DG64 & 4.5 & 1809.13 & 0.0017 & 1784.803 & 0.001 &- &- &- &- \\
		\hline
		AF73 & 2.125 & 17.4682 & 0.0525 & 17.8019 & 0.1142 &- &- &- &-  \\
		\hline
		AF74 & 2.125 & 28.7285 & 0.0114 & 29.7961 & 0.0028 &- &- &- &-  \\
		\hline
	\end{tabular}
	\caption{Overview of the results for the different methods with respect to the results for the AF54 method.}
	\label{tab:overview2}
\end{table}

\section{Conclusions}

In this work we have shown that in quite some generality, Active Flux and DG can be considered to be essentially the same method---with the crucial difference, though, that AF is of order $K+2$, and DG only of order $K+1$, unless the error is measured on certain points or some post-processing applied. One can say that Active Flux is the high-order method ``behind'' DG, that becomes visible e.g. in the Radau points. While a choice of bases (modal/nodal/etc.) in DG trivially leads to equivalent methods, the equivalence between AF and DG is more complicated to see since the spaces employed are different.

While the methods might be the same, different choices of degrees of freedom can entail very different computational cost. We have performed experiments in order to measure runtime, accuracy and memory usage in a unified implementation of both AF and DG, ensuring as much as possible their comparability. In our experiments we find Active Flux to be more efficient than DG. Especially the fifth-order Active Flux method with a fourth-order time discretisation appears to be the most accurate and cost-efficient method in the regime considered.

A study of the relationship between AF and DG for multi-dimensional nonlinear systems is going to be subject of future work.

\section*{Acknowledgement}

We acknowledge inspiring and productive discussions with Lisa Lechner and Junming Duan.

\bibliographystyle{alpha}
\newcommand{\etalchar}[1]{$^{#1}$}

\appendix

\section{Proof of Lemma \ref{lem:flux}} \label{sec:lemflux}

\begin{proof}
The starting point is \eqref{eq:vwupdateintermiedate} with $w = 1$. The defining property of $v^\pm$ is that
\begin{align}
 \frac{1}{\Delta x} \int_{-\frac{\Delta x}{2}}^{\frac{\Delta x}{2}}  v^\pm(x)  q_{ij}(x,y) \,  \dd x = q_{ij}\left(\pm\frac{\Delta x}{2}, y\right)
\end{align}
These test functions exist by Riesz' representation theorem, and they project out the corresponding values of the derivatives, because these are in the same DG space (see Remark \ref{rem:afdgspace}). Summing up \eqref{eq:vwupdateintermiedate} once with $\alpha v^+$, and once with $\beta v^-$, shifted to $i+1$, yields
\begin{align}
  \frac{\dd}{\dd t} \frac{1}{\Delta y} &\int_{-\frac{\Delta y}{2}}^{\frac{\Delta y}{2}} \left(\alpha  q_{ij}\left(\frac{\Delta x}{2},y\right) + \beta  q_{i+1,j}\left(-\frac{\Delta x}{2},y\right)\right ) \,  \dd y  
  \\\nonumber&+ \frac{1}{\Delta y} \int_{-\frac{\Delta y}{2}}^{\frac{\Delta y}{2}}  
  U^x \left(\alpha \del_x (q_{ij} + r^x_{ij})\Big |_{x=\frac{\Delta x}{2}} + \beta \del_x (q_{i+1,j} + r^x_{i+1,j})\Big |_{x=-\frac{\Delta x}{2}} \right ) \, \dd y
  \\\nonumber&+ \frac{1}{\Delta y}  U^y \left(\alpha  (q_{ij} + r^y_{ij})\Big |_{x=\frac{\Delta x}{2},y=-\frac{\Delta y}{2}}^{y=\frac{\Delta y}{2}}  + \beta  (q_{i+1,j} + r^y_{i+1,j})\Big |_{x=-\frac{\Delta x}{2},y=-\frac{\Delta y}{2}}^{y=\frac{\Delta y}{2}}  \right )
  = 0
 \end{align}

Finally, using \eqref{eq:rydef} one finds that
\begin{align}
 \left( q_{ij} + r^y_{ij} \right )\Big |_{x = \pm\frac{\Delta x}{2}, y = -\frac{\Delta y}{2}}^{y = \frac{\Delta y}{2}} &= 
 q_{ij}\left( \pm \frac{\Delta x}{2}, \frac{\Delta y}{2}\right) +  \hat q_{i,j+\frac12}\left(\pm\frac{\Delta x}{2}\right) - q_{ij}\left(\pm\frac{\Delta x}{2}, \frac{\Delta y}{2}\right)   \\\nonumber
 &- q_{ij}\left( \pm\frac{\Delta x}{2}, -\frac{\Delta y}{2}\right) - \hat q_{i,j-\frac12}\left(\pm\frac{\Delta x}{2}\right) + q_{ij}\left(\pm\frac{\Delta x}{2}, -\frac{\Delta y}{2}\right) \\
 &= \hat q_{i,j+\frac12}\left(\pm\frac{\Delta x}{2}\right)   - \hat q_{i,j-\frac12}\left(\pm\frac{\Delta x}{2}\right) 
\end{align}

Similarly, using $v=1$ and $w^\pm$ one finds the analogous expression
\begin{align}
 \frac{\dd}{\dd t} \frac{1}{\Delta x} &\int_{-\frac{\Delta x}{2}}^{\frac{\Delta x}{2}}  \left(\alpha q_{ij}\left(x, \frac{\Delta y}{2}  \right ) +\beta q_{i,j+1}\left( x, -\frac{\Delta y}{2}  \right ) \right )  \,  \dd x  
 \\\nonumber&+ \frac{1}{\Delta x}   U^x \left( \alpha   (q_{ij} + r^x_{ij}) \Big |_{y=\frac{\Delta y}{2},x=-\frac{\Delta x}{2}}^{x=\frac{\Delta x}{2}} + \beta   (q_{i,j+1} + r^x_{i,j+1}) \Big |_{y=-\frac{\Delta y}{2},x=-\frac{\Delta x}{2}}^{x=\frac{\Delta x}{2}} \right ) 
 \\\nonumber&+ \frac{1}{\Delta x} \int_{-\frac{\Delta x}{2}}^{\frac{\Delta x}{2}}  U^y \left( \alpha \del_y (q_{ij} + r^y_{ij}) \Big|_{y = \frac{\Delta y}{2}} + \beta \del_y (q_{i,j+1} + r^y_{i,j+1}) \Big|_{y = -\frac{\Delta y}{2}}  \right ) \,   \dd x  = 0 
\end{align}
which is simplified by realizing that
\begin{align}
  (q_{ij} + r^x_{ij}) \Big |_{y=\pm\frac{\Delta y}{2},x=-\frac{\Delta x}{2}}^{x=\frac{\Delta x}{2}} = 
  \hat q_{i+\frac12,j}\left(\pm\frac{\Delta y}{2}\right ) - \hat q_{i-\frac12,j}\left(\pm\frac{\Delta y}{2}\right )
\end{align}
For the third expression, one combines the two approaches and obtains immediately
\begin{align}
 \frac{\dd}{\dd t} & \left( \alpha q_{ij}\left(\frac{\Delta x}{2}, \frac{\Delta y}{2} \right) + \beta q_{i+1,j}\left(-\frac{\Delta x}{2}, \frac{\Delta y}{2} \right) + \gamma q_{i,j+1}\left(\frac{\Delta x}{2}, -\frac{\Delta y}{2} \right) + \delta q_{i+1,j+1}\left( -\frac{\Delta x}{2}, -\frac{\Delta y}{2}\right) \right )  
 \\\nonumber & + U^x \left( \alpha \del_x (q_{ij} + r^x_{ij}) \Big |_{x=\frac{\Delta x}{2},y=\frac{\Delta y}{2}} +\beta \del_x (q_{i+1,j} + r^x_{i+1,j}) \Big |_{x=-\frac{\Delta x}{2},y=\frac{\Delta y}{2}}  \right . \\ \nonumber & \qquad \qquad \left. +\gamma \del_x (q_{i,j+1} + r^x_{i,j+1}) \Big |_{x=\frac{\Delta x}{2},y=-\frac{\Delta y}{2}} +\delta \del_x (q_{i+1,j+1} + r^x_{i+1,j+1}) \Big |_{x=-\frac{\Delta x}{2},y=-\frac{\Delta y}{2}} \right )
 \\\nonumber & + U^y \left(  \alpha  \del_y (q_{ij} + r^y_{ij}) \Big |_{x=\frac{\Delta x}{2},y=\frac{\Delta y}{2}} + \beta  \del_y (q_{i+1,j} + r^y_{i+1,j}) \Big |_{x=-\frac{\Delta x}{2},y=\frac{\Delta y}{2}}  \right . \\ \nonumber & \qquad \qquad \left.+ \gamma  \del_y (q_{i,j+1} + r^y_{i,j+1}) \Big |_{x=\frac{\Delta x}{2},y=-\frac{\Delta y}{2}}  + \delta  \del_y (q_{i+1,j+1} + r^y_{i+1,j+1}) \Big |_{x=-\frac{\Delta x}{2},y=-\frac{\Delta y}{2}}  \right ) = 0 
\end{align}
\end{proof}

\section{Proof of Lemma \ref{lem:recon}} \label{sec:lemrecon}

\begin{proof}
By the fact that Radau polynomials have zero average it is clear that 
\begin{align}
 \frac{1}{\Delta x} \int_{-\frac{\Delta x}{2}}^{\frac{\Delta x}{2}}\frac{1}{\Delta y} \int_{-\frac{\Delta y}{2}}^{\frac{\Delta y}{2}} Q_{ij}(x, y) \, \dd y \dd x = \frac{1}{\Delta x} \int_{-\frac{\Delta x}{2}}^{\frac{\Delta x}{2}}\frac{1}{\Delta y} \int_{-\frac{\Delta y}{2}}^{\frac{\Delta y}{2}} q_{ij}(x, y) \, \dd y \dd x
\end{align}
and also
\begin{align}
 \frac{1}{\Delta y} &\int_{-\frac{\Delta y}{2}}^{\frac{\Delta y}{2}} Q_{ij}\left( \frac{\Delta x}{2}, y\right ) \, \dd y = \frac{1}{\Delta y} \int_{-\frac{\Delta y}{2}}^{\frac{\Delta y}{2}} q_{ij}\left( \frac{\Delta x}{2}, y\right ) + r^x_{ij}\left( \frac{\Delta x}{2}, y\right )  \, \dd y \\
 &= \frac{1}{\Delta y} \int_{-\frac{\Delta y}{2}}^{\frac{\Delta y}{2}} q_{ij}\left( \frac{\Delta x}{2}, y\right )  + \left( \hat q_{i+\frac12,j} - q_{ij}\left(\frac{\Delta x}{2}, y\right) \right ) R_{\text{R},\Delta x}^{K+1}\left(\frac{\Delta x}{2}\right)  \, \dd y \\
 &= \frac{1}{\Delta y} \int_{-\frac{\Delta y}{2}}^{\frac{\Delta y}{2}} \hat q_{i+\frac12,j}\,\dd y
\end{align}
The other edge-averages are treated analogously.
Finally, 
\begin{align}
 Q_{ij}\left( \frac{\Delta x}{2}, \frac{\Delta y}{2} \right ) &=  q_{ij}\left( \frac{\Delta x}{2}, \frac{\Delta y}{2} \right ) + r^x_{ij} \Big|_{x=\frac{\Delta x}{2}, y=\frac{\Delta y}{2} }+ r^y_{ij} \Big|_{x=\frac{\Delta x}{2}, y=\frac{\Delta y}{2} } + C_{ij}^\text{RR} &\overset{!}{=} Q_{i+\frac12,j+\frac12}
\end{align}
which can easily be solved for $C_{ij}^\text{RR}$:
\begin{align}
 C_{ij}^\text{RR} = Q_{i+\frac12,j+\frac12}  -  \hat q_{i+\frac12,j}\left(\frac{\Delta y}{2}\right)  - \hat q_{i,j+\frac12}\left(\frac{\Delta x}{2}\right) + q_{ij}\left(\frac{\Delta x}{2},\frac{\Delta y}{2}\right)\\
\end{align}
Analogously, the other three nodal values allow to fix the remaining constants $C_{ij}^\text{LL},C_{ij}^\text{LR},C_{ij}^\text{RL}$. %, e.g.
\end{proof}

\section{Proof of Lemma \ref{lem:edgeequality}} \label{sec:lemedgeequality}

\begin{proof}
 The expressions in question are second-degree polynomials. It is therefore sufficient to verify that they agree in average and in value at the endpoints. 
 
 Consider first the average of the left-hand side of \eqref{eq:identityedgex}
 \begin{align}
  \frac{1}{\Delta x} \int_{-\frac{\Delta x}{2}}^{\frac{\Delta x}{2}} \beta^+ (q_{ij} + r^x_{ij}) \Big |_{y=\frac{\Delta y}{2}} +\beta^- (q_{i,j+1} + r^x_{i,j+1}) \Big |_{y=-\frac{\Delta y}{2}} \, \dd x &= \frac{1}{\Delta x} \int_{-\frac{\Delta x}{2}}^{\frac{\Delta x}{2}} \hat q_{i,j+\frac12} \, \dd x 
 \end{align}
 while the average of the right-hand side is, by continuity of the AF approximation
 \begin{align}
  \frac{1}{\Delta x} \int_{-\frac{\Delta x}{2}}^{\frac{\Delta x}{2}}  \beta^+ Q_{ij} \Big |_{y=\frac{\Delta y}{2}} + \beta^- Q_{i,j+1}\Big |_{y=\frac{\Delta y}{2}}  \, \dd x = (\beta^+ + \beta^-) \bar Q_{i,j+\frac12}
 \end{align}
 which, with $\beta^+ + \beta^- = 1$ and \eqref{eq:identification2dedge} is indeed the same.

 Observe now that the evaluation at $x = \pm \frac{\Delta x}{2}$ allows to simplify the expressions, using \eqref{eq:rxdef}--\eqref{eq:rydef} as follows:
 \begin{align}
  \left[ \beta^+ (q_{ij} + r^x_{ij}) \Big |_{y=\frac{\Delta y}{2}} +\beta^- (q_{i,j+1} + r^x_{i,j+1}) \Big |_{y=-\frac{\Delta y}{2}} \right ]\Big|_{x = \pm \frac{\Delta x}{2}} &= \beta^+  \hat q_{i\pm\frac12,j} + \beta^- \hat q_{i\pm\frac12,j+1}  = Q_{i\pm\frac12,j+\frac12}    \\      
  \left[ \alpha^+  (q_{ij} + r^y_{ij}) \Big |_{x=\frac{\Delta x}{2}} + \alpha^-   (q_{i+1,j} + r^y_{i+1,j}) \Big |_{x=-\frac{\Delta x}{2}} \right] \Big|_{y = \pm \frac{\Delta y}{2}} &= \alpha^+ \hat q_{i,j\pm\frac12} + \alpha^- \hat q_{i+1,j\pm\frac12} = Q_{i+\frac12,j\pm\frac12}
  \end{align}
  where the last equalities are \eqref{eq:pointvalid}.
  The rest of the proof follows analogously.
\end{proof}
\end{document}